\documentclass[reqno]{amsart}
\usepackage{a4wide}
\usepackage{amsmath,amsfonts,amssymb,amsthm,version}
\usepackage{mathrsfs}
\usepackage{color}
\usepackage{epsfig}
\usepackage{subfigure}
\usepackage{enumerate}
\usepackage{graphicx}
\usepackage{subfigure}
\usepackage{rotating}
\usepackage{stmaryrd}
\usepackage{float}
\usepackage{upgreek}
\usepackage{amsopn}
\usepackage{upgreek}
\usepackage{multicol}
\usepackage{url,hyperref}
\allowdisplaybreaks

\providecommand{\U}[1]{\protect\rule{.1in}{.1in}}
\newtheorem{theorem}{Theorem}

\newtheorem{corollary}[theorem]{Corollary}

\newtheorem{definition}[theorem]{Definition}
\newtheorem{example}[theorem]{Example}

\newtheorem{lemma}[theorem]{Lemma}

\newtheorem{proposition}[theorem]{Proposition}
\newtheorem{remark}[theorem]{Remark}

\begin{document}

\title[Pugh's global linearization for  nonautonomous unbounded systems] 
{Pugh's global linearization for  nonautonomous unbounded systems with $\mu$-dichotomy via Lyapunov theory}


\author[ Weijie Lu, Yonghui Xia]
{Weijie Lu,  Yonghui Xia}  

\address{Weijie Lu, School of Mathematics  Science, Zhejiang Normal University, Jinhua  321004, China}
\email{luwj@zjnu.edu.cn}

\address{Yonghui Xia, School of Mathematics, Foshan University, Foshan 528000, China}
\email{xiadoc@163.com; yhxia@zjnu.cn. Author for correspondence}

\date{}

\begin{abstract}
  The classical global linearization theorem for autonomous system given in [C. Pugh, Amer. J. Math., 91  (1969) 363-367] requires that nonlinear system with hyperbolicity satisfies boundedness and Lipschitz continuity.
 In this paper, we establish an {\em unbounded}  global linearization theorem for nonautonomous systems subject to unbounded Lipschitz perturbations, under the assumption that the linear system admits a nonuniform $\mu$-dichotomy (more general than classical exponential dichotomy). 
 To this end, we first develop a comprehensive Lyapunov function framework for systems exhibiting nonuniform $\mu$-dichotomy. Subsequently, we establish a characterization of nonuniform $\mu$-dichotomy in terms of strict quadratic Lyapunov functions. Building upon these theoretical foundations, we then employ these Lyapunov functions to derive a linearization result under the nonuniform $\mu$-dichotomy.
 In the proof, we give a splitting lemma for nonuniform $\mu$-dichotomy to decouple   hyperbolic
system into  a contractive system   and an expansive system. Then we construct a transformation to linearize contractive/expansive system,
which is defined by the crossing time with respect to the unit sphere.
%
\end{abstract}

\maketitle

\tableofcontents

\section{Introduction}
\subsection{Motivation and contributions}
  The well-known Hartman-Grobman theorem states that $C^1$ diffeomorphism of $\mathbb{R}^n$ can be locally $C^0$ linearized near a hyperbolic equilibrium point.
By using    the methods of the structural stability of Anosov diffeomorphisms,
 Pugh (\cite{Pugh})  proved a global version of Hartman-Grobman theorem, stating that a hyperbolic isomorphism $A$ is globally topologically conjugated to  $(A+f)$ if the perturbation $f$ is bounded and small Lipschitzian. In 1973, Palmer (\cite{Palmer1})  generalized Pugh's global linearization theorem to nonautonomous systems under two crucial assumptions: first, that the linear system admits an exponential dichotomy, and second, that the nonlinear perturbation is globally bounded with sufficiently small Lipschitz constant.
Later, this theory was generalized by  (\cite{BDK-JDE,B-V2,LuK-JDDE,Lu1,LWW,ZWM-ETDS})  and the references therein.
As far as we know, the global \(C^0\) linearization results mentioned above require at least that the nonlinear perturbation is globally bounded.

 However, such globally bounded requirement imposed on the nonlinear perturbation is difficult to be satisfied.  Therefore, in this paper, we prove unbounded global linearization of hyperbolic nonautonomous systems by adopting an approach based on Lyapunov functions
and study this problem for a larger class of dynamical systems (nonuniform $\mu$-dichotomies systems).
Consider nonautonomous differential equations of the form
\begin{equation}\label{i11}
x'=A(t)x+f(t,x), \quad \text{$t\in\mathbb{R}$ and $x\in\mathbb{R}^n$,}
\end{equation}
 with the associated linear differential equation
\begin{equation}\label{i1}
x'=A(t)x, \quad \text{$t\in\mathbb{R}$ and $x\in\mathbb{R}^n$,}
\end{equation}
where $t\mapsto A(t)$ is   locally integrable and $f$ is a $C^0$-Carath\'{e}odory class function satisfies the set
\[
  \mathcal{A}_f:=\{ f(t,x)\in\mathbb{R}^n:f(t,0)=0\; \mathrm{and}\; \|f(t,x)-f(t,\tilde{x})\| \le \varphi(t) \|x-\tilde{x}\|,\; \forall t\in\mathbb{R}, x,\tilde{x}\in\mathbb{R}^n\}.
 \]
Here, \(\varphi(t) \) is a nonnegative decreasing function related to the nonuniform $\mu$-growth rate (see \eqref{ff-Lip}).
 Assume that Eq. \eqref{i1} admits a nonuniform $\mu$-dichotomy (see Definition \ref{NUD}).
Then we claim that: Eq. \eqref{i11} and Eq. \eqref{i1} are globally topological conjugacy
(see Theorem \ref{thm-lin}).

  We note that as a special case of Theorem  \ref{thm-lin}, namely, consider an autonomous system $x'=Ax+f(x)$ with $f(0)=0$,
we have the following result:
 \begin{itemize}
   \item If $A$ is hyperbolic and $f$ is a Lipschitz continuous function with a small Lipschitz constant, then there is a homeomorphism such that
 $x'=Ax+f(x)$ can be linearized globally.
 \end{itemize}
 Compared with Pugh's classical global linearization (\cite{Pugh}),
we find that the nonlinear perturbation \( f(x) \) does not require the assumption of boundedness.
Therefore, this will lead   us to overcome more technical difficulties in our proof.
 For the convenience of the readers,  we give a scheme for proving Theorem \ref{thm-lin}.
 It is known that the theory of stable/unstable manifolds and foliations plays an important role in decoupling hyperbolic
systems (see e.g., \cite{LuK-JDDE,CR-JDDE,DZZ-PLMS,Lu1,TB-JDE,Z-MA,ZWM-TAMS}).
    Regarding to our case, we first need  to  study the invariant manifolds and invariant foliations of systems with nonuniform $\mu$-dichotomy,
which can also be regarded as a generalization of the works in (\cite{BSV-JLMS,BS-JFA,CHT-JDE,CLL-JDE,ZWN-AM}).
After that, we give a splitting lemma with nonuniform $\mu$-dichotomy  to decouple the  system into a contractive system  and an
expansive  system.
Finally, inspired to the recent work (\cite{CR-DCDSA,H-DCDS,LinFX,WX-DCDS}),
 we construct a transformation to linearize contractive/expansive system, which is defined by
the crossing time with respect to the unit sphere of a strict quadratic Lyapunov function
related to a nonuniform $\mu$-contractive/expansive linear system.
We emphasize that previous studies on linearization-including (\cite{CR-DCDSA,H-DCDS,LinFX,WX-DCDS})-were limited to either contraction plus unboundedness or hyperbolic plus partial unboundedness. To the best of our knowledge, Theorem \ref{thm-lin} represents the first   attempt to achieve unbounded global linearization within the framework of  nonuniform  hyperbolic nonautonomous systems.

\subsection{Brief history of nonuniform $\mu$-dichotomy}

The concept of hyperbolicity originated from the study of phase spaces in autonomous systems
where exponential contraction and exponential expansion coexist.
 For nonautonomous systems,  to adapt this concept, in 1930 Perron (\cite{Perron})  proposed the definition of exponential dichotomy,
 i.e., assuming that there exist  (uniform) exponential contraction and exponential expansion along complementary directions
at each moment of time.
To date, the classical theory of (uniform) exponential dichotomy has reached a state of remarkable maturity,
such as  the monographs of Coppel  (\cite{Coppel}),  Sell and You (\cite{SY-book}).
In the early 1970s, the concept of nonuniform exponential dichotomy was formally introduced--a seminal departure from the classical (uniform) theory-whereby the exponential contraction/expansion rates were permitted to vary dynamically with time. 
The original references can be found in the work of Oseledets (\cite{Ose}) and Pesin (\cite{Pesin}).
For modern expositions, Barreira and Pesin (\cite{BP-book2}) as well as Barreira and Valls (\cite{B-Vbook1})
have provided detailed expositions on the current state of this theory.
In addition, the latest developments in the theory of nonuniform exponential dichotomy (particularly in applications to linearization)
can be referred to in
\cite{B-V2,B-V4,DZZ-MZ,DZZ-PLMS,Ga-JDDE24,ZLZ-JDE17} and references therein.

 Despite the foundational role of nonuniform exponential dichotomy in nonautonomous dynamics, as epitomized by Oseledets' theorem and
Pesin theory,  recent advancements prove that
there are many examples in nonautonomous dynamics that show nonexponential rates of contraction and expansion along the stable and unstable directions, we can refer to (\cite{BDZ-Arxiv,BS-JFA,JDDE-BS,DSS-CJM}) and others.

In a recent study, for linear nonautonomous differential equations, Silva (\cite{Si-JDE}) introduced the concept of nonuniform $\mu$-dichotomy,
which goes beyond the previous dichotomies.
 Notably,  the author proposed a nonuniform $\mu$-dichotomy spectrum theory associated with this new type of dichotomy and developed the reducibility and normal form theories for nonuniform $\mu$-dichotomy.
Based on this work, the theory of nonuniform $\mu$-dichotomy is now rapidly developing.
 Casta\~{n}eda and Jara (\cite{CN-JDE})
 presented a generalized Siegmund's normal form theory under the framework of nonuniform $\mu$-dichotomy spectrum theory
 and introduced concepts such as $\mu$-eligible.
Jara and Gallegos  (\cite{JG-MA}) first clarified the noninvariance of the nonuniform $\mu$-dichotomy spectrum of nonuniform kinematically similar systems, which breaks through the traditional understanding of the spectral invariance of  dichotomies.
Subsequently, Casta\~{n}eda, Gallegos  and Jara (\cite{CGJ-Arxiv})  rigorously proved that
weakly kinematically similar systems do not preserve the nonuniform dichotomy spectrum by using optimal ratio mappings.
  Dragi\v{c}evi\'{c} and  Silva  (\cite{DS-Arxiv})
 established the universal connection between nonuniform $\mu$-dichotomy and classical exponential dichotomy through time rescaling techniques and applied it to the smooth linearization theory of nonautonomous systems.

\subsection{Brief history of  Lyapunov functions}
On the other hand,  Lyapunov functions play an essential role in stability theory of  differential equations.
   This concept can be traced back to seminal work of Lyapunov in 1892 (\cite{Lyapunov}).
The earliest detailed treatment of this theory is in the monographs of LaSall and Lefschetz (\cite{LaSall}),
Bhatia and Szeg\"{o} (\cite{Bhatia}).
   But for a given differential equation,  currently,  there is not exist a general way to construct a Lyapunov function.
    Nevertheless, once we know that there is a Lyapunov function,  it will help us to  determine the  stability of the differential equation.
As a seminal application, the relationship between Lyapunov functions and (uniform) exponential dichotomy was first rigorously established by Ma\u{l}zel'  (\cite{Maizel}). A comprehensive theoretical framework connecting these concepts can be found in classical monographs by Coppel (\cite{Coppel}),  Mitropolsky, Samoilenko and Kulik (\cite{Mitropolsky}).
Subsequent advancements extended this theory to nonuniform exponential dichotomy, in which
the representative works can refer to Barreira and Valls (\cite{BV-09JDE}),  Barreira, Dragi\v{c}evi\'{c} and Valls (\cite{BDV-JDE}),
Gallegos, Grau and Mesquita (\cite{Ga-JDE21}).


\subsection{Outline}

The organization  of this paper are as follows:
 In Section 2, we introduce some notations and preliminaries, which mainly include the definition of strict and quadratic Lyapunov functions,
$\mu$-growth rate, nonuniform $\mu$-dichotomies and nonuniform $\mu$-bounded growth.
In Section 3, we   state the main results and provide the theory of invariant manifolds/foliations and Lyapunov functions in the sense of nonuniform  $\mu$-dichotomy.
In Section 4, we provide a nonlinear splitting lemma for nonuniform $\mu$-dichotomy systems.
In Section 5, we prove our  main results by using the theory of strict quadratic Lyapunov functions.
Finally,  the proofs of some useful lemmas and propositions are given in the Appendix  A-C.

\section{Notations and  preliminaries}

In this section, we recall several notions about Lyapunov functions  and nonuniform
$\mu$-dichotomies for linear differential equations.
Let $GL(n,\mathbb{R})$ denote the general linear group of order $n$ over the field of real numbers $\mathbb{R}$ and
 let $A(t)\in GL(n,\mathbb{R})$ for each $t\in\mathbb{R}$.
Consider the nonautonomous linear differential equation
\begin{align}\label{lin-eq}
x'=A(t)x, \quad t\in\mathbb{R}.
\end{align}
We suppose  that $t\mapsto A(t)$ is   locally integrable. 
As usual, we denote by $\Psi(t,s), t,s\in\mathbb{R}$, the evolution operator associated to \eqref{lin-eq}.

 \subsection{Lyapunov functions}

 Given a function $V: \mathbb{R}^{n}\rightarrow \mathbb{R}$, we consider the  cones
\[ \mathcal{C}^{u}(V)=\{0\}\cup V^{-1}(0,+\infty) \quad \mathrm{and}\quad  \mathcal{C}^{s}(V)=\{0\}\cup V^{-1}(-\infty,0).\]

\begin{definition}\label{Lp-func}{\rm
We say that a continuous function $V: \mathbb{R}\times \mathbb{R}^{n} \rightarrow \mathbb{R}$ is a  {\em Lyapunov function}
 for \eqref{lin-eq} with the notation $V_{t}=V(t,\cdot)$
if there exist
$d_s,d_u\in\mathbb{N}$ with $d_s+d_u=n$ satisfying for each $\tau\in\mathbb{R}$,
\\
(1)  $d_{u}$ and $d_{s}$ are the maximal dimensions of the linear subspaces inside $\mathcal{C}^{u}(V_{\tau})$ and $\mathcal{C}^{s}(V_{\tau})$;
\\
(2) for every $t\geq \tau$ and $x\in\mathbb{R}^{n}$,  $V(t,\Psi(t,\tau)x)\geq V(\tau,x)$.
}
\end{definition}

Given a Lyapunov function $V$, for each $\tau\in\mathbb{R}$ we define
$$
 \mathcal{E}_{\tau}^{u}=\bigcap_{t\in\mathbb{R}}\Psi(\tau,t) \overline{\mathcal{C}^{u}(V_{t})} \quad \mathrm{and} \quad
    \mathcal{E}_{\tau}^{s}=\bigcap_{t\in\mathbb{R}}\Psi(\tau,t) \overline{\mathcal{C}^{s}(V_{t})}.
$$
  Obviously, $x\in\mathcal{E}_{\tau}^{u}$ if and only if
$$
V(t,\Psi(t,\tau)x)\geq 0, \quad \forall t\in\mathbb{R},
$$
and $x\in\mathcal{E}_{\tau}^{s}$ if and only if
$$
V(t,\Psi(t,\tau)x)\leq 0, \quad \forall t\in\mathbb{R},
$$
Moreover, these sets satisfy the following property:
\begin{align}\label{E-su}
\Psi(t,\tau) \mathcal{E}_{\tau}^{u}= \mathcal{E}_{t}^{u} \quad \mathrm{and} \quad  \Psi(t,\tau) \mathcal{E}_{\tau}^{s}= \mathcal{E}_{t}^{s},
\quad \forall t,\tau \in\mathbb{R}.
\end{align}

 \subsection{$\mu$-growth rate and strict/quadratic  Lyapunov function}

\begin{definition}\label{u-groga}{\rm
We say that
a strictly increasing function $\mu:\mathbb{R}\to \mathbb{R}^+$ is a {\em growth rate} if
$$
\text{$\mu(0)=1$, $\lim_{t\to +\infty} \mu(t)=+\infty$ and $\lim_{t\to -\infty} \mu(t)=0$.}
$$
Additionally,
if $\mu$ is differentiable, it is said a {\em differentiable growth rate}.
}
\end{definition}

Let $V$ be a Lyapunov function for Eq. \eqref{lin-eq}, and suppose that there exist $C>0$ and $\epsilon \ge 0$ such that
\begin{align}\label{u-grow}
 |V(\tau,x)| \le C \mu(\tau)^{\mathrm{sign}(\tau)\epsilon}\|x\|, \quad \text{$\forall (\tau,x)\in \mathbb{R}\times \mathbb{R}^n$.}
\end{align}

\begin{definition}\label{str-V}
{\rm
We say that $V$ is a {\em strict Lyapunov function} if there exist $C>0$, $\epsilon\ge 0$ and $\alpha<0, \beta<0$ such that for each $\tau\in\mathbb{R}$,
\\
(1) if $x\in \mathcal{E}_{\tau}^{u}$, then
$$
V(t,\Psi(t,\tau)x)\ge   \left(\frac{\mu(\tau)}{\mu(t)}\right)^\alpha  V(\tau,x), \quad  t\ge  \tau;
$$
(2) if $x\in \mathcal{E}_{\tau}^{s}$, then
$$
|V(t,\Phi(t,\tau)x)| \leq  \left(\frac{\mu(t)}{\mu(\tau)}\right)^\beta  |V(\tau,x)|, \quad  t\ge \tau;
$$
(3) if $x\in \mathcal{E}_{\tau}^{s}\cup \mathcal{E}_{\tau}^{u}$, then
$$
  |V(\tau,x)| \ge \mu(\tau)^{-\mathrm{sign}(\tau)\epsilon}\|x\|/C.
$$
}
\end{definition}

\begin{remark}{\rm
Notice that $\mu(\tau)^{-\mathrm{sign}(\tau)\epsilon}\in (0,1]$ for all $\tau\in\mathbb{R}$.
}
\end{remark}


\begin{definition}\label{quad-V}{\rm
We say that  $V$ is a {\em quadratic Lyapunov function} if
\begin{equation*}
   U(t,x)= \langle S(t)x, x \rangle \quad \mathrm{and} \quad V(t,x)=- \mathrm{sign}\, U(t,x) \sqrt{|U(t,x)|},
\end{equation*}
where $\langle\cdot,\cdot\rangle$ is the usual inner product of $\mathbb{R}^{n}$ and $S(t)\in GL(n,\mathbb{R})$ is
 symmetric invertible for each $t\in\mathbb{R}$.
}
\end{definition}

\begin{remark}\label{Rem-4}{\rm
Given  $A,B\in GL(n,\mathbb{R})$, we write $A\leq B$ if they verify $\langle Ax,x\rangle \leq \langle Bx,x\rangle$ for $x\in\mathbb{R}^{n}$.
}
\end{remark}

\begin{remark}\label{Rem-41}{\rm
Let $S(t)$ be differentiable. In the case of {\em differentiable quadratic Lyapunov functions}, we define
\begin{align*}
\dot{V}(t,x)=&~ \frac{d}{d\hbar}V(t+\hbar,\Psi(t+\hbar,t)x)|_{\hbar=0}
\\
=&~ \frac{\partial}{\partial t}V(t,x)+\frac{\partial}{\partial x}V(t,x)A(t)x
\\
=&~  \langle S'(t)x, x \rangle + 2  \langle S(t)x, A(t)x \rangle.
\end{align*}
}
\end{remark}


\subsection{Nonuniform  $\mu$-dichotomies}

\begin{definition}\label{NUD}{\rm
Let $\mu:\mathbb{R}\to \mathbb{R}^+$ be a growth rate.
We say that Eq. \eqref{lin-eq} admits a {\em nonuniform $\mu$-dichotomy}
 if there exists a family of projections $\pi(t)\in GL(n,\mathbb{R})$, $t\in\mathbb{R}$, such that
$$
\pi(t)\Psi(t,s)=\Psi(t,s)\pi(s), \quad \forall t,s\in\mathbb{R},
$$
and there are constants $D\ge 1$, $\lambda_s<0$, $\lambda_u>0$ and $\nu,\omega\ge 0$, with $\lambda_s+\nu<0$ and $\lambda_u-\omega>0$, satisfying
\begin{equation}\label{u-dich}
\begin{cases}
 \|\Psi(t,s)\pi(s) \| \leq D \left( \frac{\mu(t)}{\mu(s)} \right)^{\lambda_s} \mu(s)^{\mathrm{sign}(s)\nu}, \quad \text{for}  \; t\ge s,
\\
 \|\Psi(t,s)(id-\pi(s)) \| \leq D  \left( \frac{\mu(t)}{\mu(s)} \right)^{\lambda_u} \mu(s)^{\mathrm{sign}(s)\omega}, \quad \text{for}  \;  t\le s,
\end{cases}
\end{equation}
where $id$ denotes the identity operator.
}
\end{definition}

In this case, for each $\tau \in\mathbb{R}$, we define the stable and unstable subspaces by
$$
\mathcal{X}_\tau^s=\pi(\tau)(\mathbb{R}^n), \quad  \mathcal{X}_\tau^u=(id-\pi(\tau))(\mathbb{R}^n).
$$
Furthermore, if $\nu=\omega=0$, then we say that Eq. \eqref{lin-eq} admits a uniform $\mu$-dichotomy.

\begin{example}{\rm
$\bullet$ Letting $\mu(t)=e^{t}$, $-\lambda_s=\lambda_u$ and $\nu=\omega=0$, then we recover the classical exponential dichotomy defined by Perron (\cite{Perron}.
\\
$\bullet$ Taking just $\mu(t)=e^t$, we obtain the notion of nonuniform exponential dichotomy.  It originated from Pesin (\cite{Pesin})
and has been widely studied by (\cite{BP-book2}).
\\
$\bullet$ Given a strictly increasing surjective function $\chi:[0,+\infty)\to [1,+\infty)$, we can define a growth rate $\mu:\mathbb{R}\to\mathbb{R}$ by
\begin{equation}\label{u-pogr}
\mu(t):=
\begin{cases}
 \chi(t), \quad \text{if $t\ge 0$},
\\
\frac{1}{\chi(|t|)}, \quad \text{if $t\le 0$},
\end{cases}
\end{equation}
and obtain the  corresponding nonuniform $\mu$-dichotomy.
Note that if $\chi(t)=t+1$, we obtain the growth rate $p(t)$ and the corresponding nonuniform polynomial dichotomy notion, see e.g., \cite{BDZ-Arxiv,BS-JFA,DSS-CJM}.
}
\end{example}

\subsection{Nonuniform $\mu$-bounded growth}

We  present a concept that describes the nonuniform $\mu$-bounded growth of the evolution operator $\Psi(t,s)$, $t,s\in\mathbb{R}$,
which was introduced by \cite[pp. 630-631]{Si-JDE}.

\begin{definition}\label{def-bogr}{\rm
We say that Eq. \eqref{lin-eq} admits a {\em nonuniform $\mu$-bounded growth}
 if there exist $D\ge 1, \lambda_{\max}>0$ and $\theta\ge 0$ such that for all $t,s\in\mathbb{R}$,
\begin{align}\label{u-bogr}
\|\Psi(t,s)\| \le D \left(\frac{\mu(t)}{\mu(s)}\right)^{\mathrm{sign}(t-s)\lambda_{\max}} \mu(s)^{\mathrm{sign}(s)\theta}.
\end{align}
}
\end{definition}

\begin{example}{\rm
If   $\mu(t)=e^t$,  we say that Eq. \eqref{lin-eq} exhibits a nonuniform exponential bounded growth. 
 If $\mu(t)$ has the form \eqref{u-pogr}, then Eq. \eqref{lin-eq} exhibits a nonuniform polynomial bounded growth. 
}
\end{example}

\section{Main results and auxiliary lemmas}

\subsection{Main results}
We firstly give a notion of a Carath\'{e}odory class function, which was introduced in \cite[Section 2.2]{CN-JDE}.
Let $\Omega \subset  \mathbb{R}\times \mathbb{R}^n$.
We say that $f:\Omega\to \mathbb{R}^n$ is a Carath\'{e}odory class function if for every  $\mathbb{I}\times U\subset \Omega$,
(i) $f(t,\cdot):U\to\mathbb{R}^n$ is continuous for almost all $t\in \mathbb{I}$ (i.e., outside  of a set of zero Lebesgue measure);
(ii) $f(\cdot,x):\mathbb{I}\to \mathbb{R}^n$ is measurable for each $x\in U$.

Consider a nonlinear perturbation of Eq. \eqref{lin-eq} as follows:
\begin{equation}\label{nonlin-eq}
x'=A(t)x+f(t,x), 
\end{equation}
where $f$ is a $C^0$-Carath\'{e}odory class function such that $f(t,0)=0$.
Let
\begin{align}\label{ff-Lip}
\varphi(t):=\delta_f \mu(t)^{-1-\mathrm{sign}(t)\theta} \mu'(t)
\end{align}
for every  $t\in\mathbb{R}$ with a small constant $\delta_f>0$.
Define a set
\[
  \mathcal{A}_f:=\{ f(t,x)\in\mathbb{R}^n:f(t,0)=0\; \mathrm{and}\; \|f(t,x)-f(t,\tilde{x})\| \le \varphi(t) \|x-\tilde{x}\|,\; \forall t\in\mathbb{R}, x,\tilde{x}\in\mathbb{R}^n\}.
 \]

Then we have the following result.

\begin{lemma}\label{non-esti}
Let $x(t,\tau,x_0)$ be the solution of Eq. \eqref{nonlin-eq} such that  \eqref{u-bogr} holds and $f(t,x)$ satisfies the set $\mathcal{A}_f$. Then for all $t,\tau\in\mathbb{R}$
and $x_0,\tilde{x}_0\in\mathbb{R}^n$, the following estimates hold:
$$
\|x(t,\tau,x_0)-x(t,\tau,\tilde{x}_0)\| \ge
\frac{1}{D} \left(\frac{\mu(t)}{\mu(\tau)}\right)^{-\mathrm{sign}(t-s)(\lambda_{\max}+D\delta_f)}\mu(\tau)^{-\mathrm{sign}(\tau)\theta}
\|x_0-\tilde{x}_0\|
$$
and
$$
\|x(t,\tau,x_0)-x(t,\tau,\tilde{x}_0)\| \le  D
  \left(\frac{\mu(t)}{\mu(\tau)}\right)^{ \mathrm{sign}(t-s)(\lambda_{\max}+D\delta_f)}\mu(\tau)^{\mathrm{sign}(\tau)\theta}\|x_0-\tilde{x}_0\|.
$$
\end{lemma}

We put the proof of this lemma in Appendix \ref{App-2}.
Next we are in a position to present our main results in this paper.

\begin{theorem}\label{thm-lin}
Assume that Eq. \eqref{nonlin-eq} admits a nonuniform $\mu$-dichotomy and that $f(t,x)$ satisfies the set $\mathcal{A}_f$.
Then if
$\theta\ge  \max\{\nu,\omega\}$, $\lambda_s<\nu-\theta$ and $\lambda_u>\theta-\omega$, then
Eq. \eqref{nonlin-eq} is topologically conjugated to Eq. \eqref{lin-eq}.
\end{theorem}

\begin{corollary}
Assume that Eq. \eqref{nonlin-eq} admits a nonuniform  exponential dichotomy (i.e., $\mu(t)=e^t$) and that $f(t,x)$ satisfies
  the set $\mathcal{A}_f$ with $\varphi(t)=\delta_f e^{-\theta |t|}$.
Then if
$\theta\ge   \max\{\nu,\omega\}$, $\lambda_s<\nu-\theta$ and $\lambda_u>\theta-\omega$, then
Eq. \eqref{nonlin-eq} is topologically conjugated to Eq. \eqref{lin-eq}.
\end{corollary}


A special case of  \eqref{nonlin-eq} is the autonomous system
\begin{align}\label{Auto-NL}
x'=Ax+f(x), \quad x\in\mathbb{R}^n,
\end{align}
where $A$ is a $n\times n$ constant matrix and $f(0)=0$. Then by Theorem \ref{thm-lin}, we obtain the following.

\begin{corollary}\label{Cor-GP} {\rm (Generalized Pugh's linearization theorem)}
Let the matrix $A$   be hyperbolic and let $f$ be Lipschitz continuous with a small Lipschitz constant $\delta_f$.
Then Eq. \eqref{Auto-NL} is topologically conjugated to its linear part.
\end{corollary}

We remark that the difference between Pugh's global $C^0$ linearization theorem and Corollary \ref{Cor-GP} is that
our result does not require $f$ to be a bounded function.
By using    the methods of the structural stability of Anosov diffeomorphisms, Pugh (\cite{Pugh})  proved  a global linearization theorem.
Later, this theory was generalized by \cite{BDK-JDE,B-V2,Lu1,Palmer1,ZWM-ETDS}  and the references therein.
As far as we know, the global \(C^0\) linearization results mentioned above require at least that the nonlinear perturbation is globally bounded.
Recently, Lin (\cite{LinFX}), Huerta (\cite{H-DCDS}), Casta\~{n}eda and Robledo (\cite{CR-DCDSA}) respectively
established the global unbounded linearization  results for nonuniform (or uniform) contractive systems, while Wu an Xia  (\cite{WX-DCDS})  proved a partially unbounded linearization result  for uniform hyperbolic nonautonomous systems.
However, the unbounded linearization of hyperbolic systems (especially nonuniform hyperbolic systems) has not yet been completely resolved.

As stated in the Introduction, in order to overcome the technical difficulties in the proof caused by the lack of boundedness,
we first need to decouple the system into a contraction system and an expansion system by combining the theory of invariant manifolds/foliations.
Then, we construct a transformation to linearize contractive/expansive system, which is defined by
the crossing time with respect to the unit sphere of a strict quadratic Lyapunov function
related to a nonunifrom $\mu$-contractive/expansive linear system.

To achieve the two steps mentioned above, in the following, we first establish the existence of global stable and unstable invariant manifolds and invariant foliations for Eq. \eqref{nonlin-eq}, and then we  establish the relationship between strict and/or  quadratic Lyapunov functions and nonuniform $\mu$-dichotomy.

\subsection{Invariant manifolds and invariant foliations for nonuniform $\mu$-dichotomies}

Let Eq. \eqref{lin-eq} admit  a nonuniform $\mu$-dichotomy.
Notice that $\mathcal{X}_\tau^s=\pi(\tau)(\mathbb{R}^n)$ and $\mathcal{X}_\tau^u=(id-\pi(\tau))(\mathbb{R}^n)$  are the stable and unstable subspaces at time $\tau$,
respectively.
Then we have the following results.

\begin{lemma}\label{thm-stab}
Assume that Eq. \eqref{nonlin-eq} admits a nonuniform $\mu$-dichotomy and that $f(t,x)$ satisfies the set $\mathcal{A}_f$ with a small constant $\delta_f>0$.
Then if
$\theta\ge   \max\{\nu,\omega\}$, $\lambda_s<\nu-\theta$ and $\lambda_u>\theta-\omega$, then
Eq. \eqref{nonlin-eq} exhibits   global Lipschitz stable and unstable invariant manifolds, which can be described  as
$$
\text{$\mathcal{G}_s=\{\xi+g_s(\tau,\xi): \tau\in\mathbb{R},\xi\in\mathcal{X}_\tau^s\}$ \,  and \,  $\mathcal{G}_u=\{\xi+g_u(\tau,\xi): \tau\in\mathbb{R},\xi\in\mathcal{X}_\tau^u\}$,}
$$
where $g_s:\mathbb{R}\times \mathcal{X}_\tau^s \to \mathcal{X}_\tau^u$ and  $g_u:\mathbb{R}\times \mathcal{X}_\tau^u \to \mathcal{X}_\tau^s$
are Lipschitz continuous mappings satisfying $g_s(\tau,0)=0$ and $g_u(\tau,0)=0$.
\end{lemma}

From Lemma \ref{thm-stab}, it follows that there exist a stable manifold $\mathcal{G}_s$ and an unstable manifold $\mathcal{G}_u$ for Eq. \eqref{nonlin-eq}.
They are graphs of $C^{0,1}$  functions  $g_s(t,\cdot):  \mathcal{X}_\tau^s \to \mathcal{X}_\tau^u$ and $g_u(t,\cdot): \mathcal{X}_\tau^u \to \mathcal{X}_\tau^s$
respectively satisfying $g_s(t,0)=g_u(t,0)=0$ for all $t\in\mathbb{R}$.
Then by using two $C^{0,1}$ transformations
$$
\text{$\Xi_s(t,x_s,x_u)=(t,x_s,x_u-g_s(t,x_s))$ \, and \, $\Xi_u(t,x_s,x_u)=(t,x_s-g_u(t,x_u),x_u)$,}
$$
one can straighten up the stable and unstable manifolds. This enable us to assume that
\begin{align*}
\text{$\pi(t)f(t,(id-\pi(t))x(t))=0$ \, and \, $(id-\pi(t))f(t,\pi(t)x(t))=0$.}
\end{align*}

\begin{lemma}\label{thm-foli}
Assume that Eq. \eqref{nonlin-eq} admits a nonuniform $\mu$-dichotomy and that $f(t,x)$ satisfies the set $\mathcal{A}_f$ with a small constant $\delta_f>0$.
Then if
$\theta\ge   \max\{\nu,\omega\}$, $\lambda_s<\nu-\theta$ and $\lambda_u>\theta-\omega$, then
Eq. \eqref{nonlin-eq} exhibits   global   stable and unstable invariant foliations,  whose leaves
defined as
$$
\text{$\mathcal{W}_s(\tau,x)=\{\zeta+h_s(\tau,\zeta,x): \tau\in\mathbb{R},\zeta \in\mathcal{X}_\tau^s\}$ \,  and \,  $\mathcal{W}_u(\tau,x)=\{\zeta+h_u(\tau,\zeta,x): \tau\in\mathbb{R},\zeta \in\mathcal{X}_\tau^u\}$,}
$$
where   $h_s:\mathbb{R}\times \mathcal{X}_\tau^s \times \mathbb{R}^n \to \mathcal{X}_\tau^u$ and
$h_u:\mathbb{R}\times \mathcal{X}_\tau^u \times \mathbb{R}^n \to \mathcal{X}_\tau^s$
are  continuous in both variables $(\zeta,x)$ and are Lipschitz continuous in variable $\zeta$, respectively.
\end{lemma}

We will employ the classical Lyapunov-Perron method to prove the existence of the stable/unstable manifold and stable/unstable foliation,
which has been repeatedly mentioned in the  recent literature such as \cite{BSV-JLMS,BS-JFA,JDDE-BS,ZWN-AM}.
Thus, we put the proof of  Lemmas \ref{thm-stab} and \ref{thm-foli} in Appendix \ref{App-3}.

\subsection{Lyapunov functions and nonuniform $\mu$-dichotomies}

In this subsection, we study the relation between nonuniform $\mu$-dichotomy and strict quadratic Lyapunov functions.
More precisely,  we obtain the existence of a strict Lyapunov function  provided that Eq. \eqref{lin-eq} admits a nonuniform $\mu$-dichotomy
 (see Proposition \ref{V1-thm}).
Moreover, under the assumption of bounded growth, the existence of  a strict quadratic Lyapunov function is further established
(see Proposition \ref{V-thm2}).
 In particular, by constructing a explicit   strict quadratic Lyapunov function, it is proved that Eq. \eqref{lin-eq} implies nonuniform $\mu$-dichotomy
(see Proposition \ref{thm-ud}).

\begin{proposition}\label{V1-thm}
Suppose that Eq. \eqref{lin-eq} admits a nonuniform $\mu$-dichotomy. Then it exists a strict Lyapunov function.
\end{proposition}
%

\begin{proposition}\label{V-thm2}
Suppose that Eq. \eqref{lin-eq} admits a nonuniform $\mu$-dichotomy. If  
there are constants $\tilde{D}\ge 1$, $c>0$ and $\tilde{\lambda}>0$
such that
\begin{align}\label{bou-u}
\|\Psi(t,s)\| \le \tilde{D} \mu(|t|)^{\tilde{\lambda}}, \quad  \forall t\in\mathbb{R},s\in[t-c,t+c],
\end{align}
then there is a  strict quadratic Lyapunov function for Eq. \eqref{lin-eq}.
\end{proposition}

Remark that if  $A(t)$ is bounded,  then property \eqref{bou-u} holds.
Let $S(t)\in GL(n,\mathbb{R})$ be symmetric invertible satisfying the form
\begin{align*}
S(t)=&~ \int_t^{+\infty} \big(\Psi(\kappa,t)\pi(t)\big)^{*} \Psi(\kappa,t)\pi(t)
                \left(\frac{\mu(\kappa)}{\mu(t)}\right)^{-2(\lambda_s+\eta)} \frac{\mu'(\kappa)}{\mu(\kappa)} \, d\kappa
\notag \\
&~ - \int_{-\infty}^t  \big(\Psi(\kappa,t)(id-\pi(t))\big)^{*} \Psi(\kappa,t)(id-\pi(t))
             \left(\frac{\mu(t)}{\mu(\kappa)}\right)^{2(\lambda_u-\eta)} \frac{\mu'(\kappa)}{\mu(\kappa)} \, d\kappa,
\end{align*}
for some constant $\eta>0$ such that $\eta<\min\{-\lambda_s,\lambda_u\}$.
Then we present some properties of the operator $S(t)$, which will provide useful estimates for Theorem \ref{thm-lin}.

\begin{proposition}\label{Pro18}
Suppose that Eq. \eqref{lin-eq} admits a nonuniform $\mu$-dichotomy satisfying \eqref{bou-u}.
Then the following properties hold:
\begin{align}
\|S(t)\| \le \frac{D^2}{\eta} \max\left\{\mu(t)^{\mathrm{sign}(t)\nu},  \mu(t)^{\mathrm{sign}(t)\omega}\right\},
\label{S-a} \\
 S'(t)+A(t)^*S(t)+S(t)A(t)  \le -\frac{2\mu'(t)}{\mu(t)} (id+\lambda_u S(t)),
\label{S-b}
\end{align}
and
\begin{align}
\text{$S(t)\ge \frac{C_s}{4}\mu(|t|)^{-2\tilde{\lambda}}$ \, $\mathrm{on}$ $\mathcal{E}_t^s$, \quad  $-S(t)\ge \frac{C_u}{4}\mu(|t|)^{-2\tilde{\lambda}}$ \, $\mathrm{on}$  $\mathcal{E}_t^u$.}
\label{S-c}
\end{align}
\end{proposition}

\begin{remark}\label{Rem-5}{\rm
Let $x(\kappa)=\Psi(\kappa,\tau)x(\tau)$ for all $\kappa,\tau\in\mathbb{R}$. It follows from \eqref{S-b} that
\begin{align*}
 \frac{d}{d\kappa} U(\kappa,x(\kappa))=& \langle S'(\kappa)x(\kappa),x(\kappa) \rangle
         + \langle S(\kappa)x'(\kappa),x(\kappa) \rangle + \langle S(\kappa)x(\kappa), x'(\kappa) \rangle
\\
=&  \langle (S'(\kappa)+S(\kappa)A(\kappa)+A(\kappa)^* S(\kappa)) x(\kappa),x(\kappa) \rangle
\\
\le&  -\frac{2\mu'(\kappa)}{\mu(\kappa)}\|x(\kappa)\|^2 -\frac{2\lambda_u \mu'(\kappa)}{\mu(\kappa)}U(\kappa,x(\kappa)).
\end{align*}
}
\end{remark}

Based on estimate \eqref{u-grow} and the definition of strict Lyapunov function, we claim the   result.

\begin{proposition}\label{thm-str}
Let  $\mathcal{G}_\tau^s \subset \mathcal{E}_\tau^s$ and $\mathcal{G}_\tau^u \subset \mathcal{E}_\tau^u$ be  two subspaces satisfying
$$
\text{$\Psi(t,\tau) \mathcal{G}_\tau^s=\mathcal{G}_t^s$ \, and \, $\Psi(t,\tau) \mathcal{G}_\tau^u=\mathcal{G}_t^u$, $\forall t,\tau\in\mathbb{R}$.}
$$
Replace \(\mathcal{E}_\tau^s\) and \(\mathcal{E}_\tau^u\) in Definition \ref{str-V} with \(\mathcal{G}_\tau^s\) and \(\mathcal{G}_\tau^u\), then for any $t\ge \tau$, we have
\begin{equation}\label{u-lem}
\begin{cases}
\|\Psi(t,\tau)|\mathcal{G}_\tau^s\| \le
   C^2 \left( \frac{\mu(t)}{\mu(\tau)}\right)^{\beta+\mathrm{sign}(t)\epsilon}\mu(\tau)^{2\mathrm{sign}(\tau)\epsilon},
\\
\|\Psi(t,\tau)^{-1}|\mathcal{G}_t^u\| \le C^2 \left(\frac{\mu(\tau)}{\mu(t)}\right)^{-\alpha+\mathrm{sign}(\tau)\epsilon}
  \mu(t)^{ 2\mathrm{sign}(t) \epsilon}.
\end{cases}
\end{equation}
\end{proposition}

Next, in the case of strict quadratic Lyapunov functions, it is proved that Eq. \eqref{lin-eq} implies nonuniform $\mu$-dichotomy.

\begin{proposition}\label{thm-ud}
Suppose that $S(t)\in GL(n,\mathbb{R})$ is symmetric invertible satisfying \eqref{S-a} and \eqref{S-c}.
If Eq. \eqref{lin-eq} has a strict quadratic Lyapunov function, then it admits a nonuniform $\mu$-dichotomy.
\end{proposition}

The relationship between nonuniform exponential dichotomy and Lyapunov functions has   been established in \cite{BDV-JDE,BV-09JDE}.
In order to allow readers to quickly see the proof of our main results, we  put the proof of  Propositions \ref{V1-thm}--\ref{thm-ud} in Appendix \ref{App-4}.

\section{The splitting lemma for nonuniform $\mu$-dichotomies}

In this section, we introduce a nonlinear    splitting lemma for  nonuniform $\mu$-dichotomies,
 which is used to decouple Eq. \eqref{nonlin-eq} into a contraction system and an expansion system by straightening up the invariant
foliations, i.e.,
the following result holds.

\begin{lemma}\label{lem-spli}
Assume that Eq. \eqref{nonlin-eq} admits a nonuniform $\mu$-dichotomy and that $f(t,x)$ satisfies the set $\mathcal{A}_f$ with a small constant $\delta_f>0$.
Then if
$\theta\ge   \max\{\nu,\omega\}$, $\lambda_s<\nu-\theta$ and $\lambda_u>\theta-\omega$, then
 there is a homeomorphism $\mathfrak{S}(t,\cdot):\mathbb{R}^n \to \mathbb{R}^n$
such that Eq. \eqref{nonlin-eq} is conjugated to the system
\begin{equation}\label{dec-eqs}
\begin{cases}
 x_s'(t)=A_s(t)x_s(t)+\pi(t)f(t,x_s(t)),
\\
 x_u'(t)=A_u(t)x_u(t)+(id-\pi(t))f(t,x_u(t)).
\end{cases}
\end{equation}
\end{lemma}

\begin{proof}[Proof of  Lemma \ref{lem-spli}.]
For any $t\in\mathbb{R}$, define $\mathfrak{S}(t,\cdot):\mathbb{R}^n\to \mathbb{R}^n$ and
$\hat{\mathfrak{S}}(t,\cdot):\mathbb{R}^n\to \mathbb{R}^n$ by
\begin{align}
\mathfrak{S}(t,x)=&~h_s(t,0,x)+h_u(t,0,x),
\label{SS-1} \\
\hat{\mathfrak{S}}(t,x)=&~ h_s(t,\pi(t)\hat{\mathfrak{S}}(t,x),(id-\pi(t))x)+h_u(t,(id-\pi(t))\hat{\mathfrak{S}}(t,x),\pi(t)x).
\label{SS-2}
\end{align}
Form  Lemma \ref{thm-foli}, we see that $\mathfrak{S}(t,x)$ is well-defined (where $\zeta=0$ in $h_s$ and $h_u$).
Define an operator $\mathcal{L}:\mathbb{R}^n \to \mathbb{R}^n$ by
$$
(\mathcal{L}\hat{\mathfrak{S}})(t,x)=h_s(t,\pi(t)\hat{\mathfrak{S}}(t,x),(id-\pi(t))x)+h_u(t,(id-\pi(t))\hat{\mathfrak{S}}(t,x),\pi(t)x),
\quad \forall t\in\mathbb{R}.
$$
Since $h_s$ and $h_u$ are Lipschitz continuous in $\zeta$ with small Lipschitz constants
respectively, one can directly obtain that $\mathcal{L}$ is a contraction mapping.
Hence, we affirm that there is a unique $\hat{\mathfrak{S}}$ satisfying \eqref{SS-2}.
Furthermore, $\mathfrak{S}(t,x)$ and $\hat{\mathfrak{S}}(t,x)$ are both continuous functions since the continuity of the stable and unstable foliations.

Next, we show that $\mathfrak{S}(t,\hat{\mathfrak{S}}(t,x))=\hat{\mathfrak{S}}(t,\mathfrak{S}(t,x))=x$.
Indeed, choose a $x\in\mathbb{R}^n$, we have
\begin{align}\label{Sp-1}
h_s(t,\pi(t)x,x)+h_u(t,(id-\pi(t))x,x)=x, \quad \forall t\in\mathbb{R}.
\end{align}
In addition, for arbitrary $y\in\mathbb{R}^n$,
\begin{align}\label{Sp-2}
\text{if  \, $h_s(t,\pi(t)y,x)=(id-\pi(t))y$,  \, then \,  $h_s(t,\cdot,x)=h_s(t,\cdot,y)$,   }
\end{align}
and
\begin{align}\label{Sp-3}
\text{if  \, $h_u(t,(id-\pi(t))y,x)=\pi(t)y$,  \, then \,  $h_u(t,\cdot,x)=h_u(t,\cdot,y)$.   }
\end{align}
In view of \eqref{SS-1}, we have
$$
(id-\pi(t))\mathfrak{S}(t,x)=h_s(t,0,x), \quad \pi(t)\mathfrak{S}(t,x)=h_u(t,0,x).
$$
Then, we obtain from \eqref{Sp-2} and \eqref{Sp-3} that
$$
h_s(t,\cdot,x)=h_s(t,\cdot,(id-\pi(t))\mathfrak{S}(t,x)), \quad h_u(t,\cdot,x)=h_u(t,\cdot,\pi(t)\mathfrak{S}(t,x)).
$$
Thus,
\begin{align}\label{Sp-4}
x=&~ h_s(t,\pi(t)x,x)+h_u(t,(id-\pi(t))x,x)
\notag \\
=&~ h_s(t,\pi(t)x,(id-\pi(t))\mathfrak{S}(t,x))+h_u(t,(id-\pi(t))x,\pi(t)\mathfrak{S}(t,x)).
\end{align}
Moreover, for all $t\in\mathbb{R}$,
\begin{align*}
\hat{\mathfrak{S}}(t,\mathfrak{S}(t,x))=&~ h_s(t,\pi(t)\hat{\mathfrak{S}}(t,\mathfrak{S}(t,x)),(id-\pi(t))\mathfrak{S}(t,x))
\\
&~ + h_u(t,(id-\pi(t))\hat{\mathfrak{S}}(t,\mathfrak{S}(t,x)),\pi(t)\mathfrak{S}(t,x) ),
\end{align*}
which together with \eqref{Sp-1}, \eqref{Sp-4} yields that
$$
\hat{\mathfrak{S}}(t,\mathfrak{S}(t,x))=x, \quad \forall t\in\mathbb{R}.
$$
On the other hand, we have
\begin{align*}
&h_s(t,\pi(t)(id-\pi(t))x,\hat{\mathfrak{S}}(t,x))=h_s(t,0,\hat{\mathfrak{S}}(t,x))=(id-\pi(t))x,
\\
& h_u(t,(id-\pi(t))\pi(t)x,\hat{\mathfrak{S}}(t,x))=h_u(t,0,\hat{\mathfrak{S}}(t,x))=\pi(t)x.
\end{align*}
The above two equality imply that
$$
\mathfrak{S}(t,\hat{\mathfrak{S}}(t,x))=h_s(t,0,\hat{\mathfrak{S}}(t,x))+h_u(t,0,\hat{\mathfrak{S}}(t,x))=x, \quad \forall t\in\mathbb{R}.
$$

  In what follows, we define a nonlinear operator $\mathcal{N}(t):\mathbb{R}\times \mathbb{R}^n \to \mathbb{R}^n$ by
$$
\mathcal{N}(t)(\tau,x):=x(t,\tau,x)=\Psi(t,\tau)x+\int_{\tau}^t \Psi(t,\kappa) f(\kappa,\mathcal{N}(\kappa)(\tau,x)) \, d\kappa.
$$
For any $x,y\in\mathbb{R}^n$, we see that the operator $\mathcal{N}(t)$, $t\in\mathbb{R}$, satisfies the following properties:
\begin{description}
  \item[(IP.1)] letting $id_s$ and $id_u$ be the identity operators on $\mathcal{X}_\tau^s$ and $\mathcal{X}_\tau^u$, respectively,  then
$$
\mathcal{N}(t)(\tau,\cdot)\circ id_s =\pi(t) \circ \mathcal{N}(t)(\tau,\cdot)\circ id_s=: \mathcal{N}^s(t)(\tau,\pi(\tau)\cdot)
$$
and
$$
\mathcal{N}(t)(\tau,\cdot)\circ id_u=(id-\pi(t))\circ \mathcal{N}(t) (\tau,\cdot)\circ id_u =: \mathcal{N}^u(t)(\tau,(id-\pi(\tau))\cdot).
$$
  \item[(IP.2)] for any $y_s=\pi(\tau)y \in \mathcal{X}_\tau^s$ and $y_u=(id-\pi(\tau))y\in \mathcal{X}_\tau^u$,
\begin{align*}
\mathcal{N}(t)(\tau,y_s+h_s(\tau,y_s,x)) =&~ \pi(t)\mathcal{N}(t)(\tau,y_s+h_s(\tau,y_s,x))
\\
=&~ h_s(t,\pi(t)\mathcal{N}(t)(\tau,y_s+h_s(\tau,y_s,x)),\mathcal{N}(t)(\tau,x))
\end{align*}
and
\begin{align*}
\mathcal{N}(t)(\tau,y_u+h_s(\tau,y_u,x)) =&~ (id-\pi(t)) \mathcal{N}(t)(\tau,y_u+h_u(\tau,y_u,x))
\\
=&~ h_u(t,(id-\pi(t))\mathcal{N}(t)(\tau,y_u+h_u(\tau,y_u,x)), \mathcal{N}(t)(\tau,x)).
\end{align*}
\end{description}

Then we have the following lemma.
\begin{lemma}\label{split-lem}
For any $t,\tau\in\mathbb{R}$ and $x\in\mathbb{R}^n$, if  {\bf (IP.1)} and {\bf (IP.2)} hold, then
\begin{align}
\mathfrak{S}(t,\mathcal{N}(t)(\tau,x))
=&~ \mathcal{N}^s(t)(\tau,\pi(\tau)\mathfrak{S}(\tau,x)) + \mathcal{N}^u(t)(\tau,(id-\pi(\tau)) \mathfrak{S}(\tau,x))
\label{Sp-5} \\
\mathcal{N}(t)(\tau,\hat{\mathfrak{S}}(\tau,x))
=&~ \hat{\mathfrak{S}}(t,\cdot)\circ (\mathcal{N}^s(t)(\tau,\pi(\tau)x)+\mathcal{N}^u(t)(\tau,(id-\pi(\tau))x))).
\label{Sp-6}
\end{align}
\end{lemma}

\begin{proof}[Proof of Lemma \ref{split-lem}.]
It follows from \eqref{SS-1}, \eqref{SS-2} and {\bf (IP.1)} and {\bf (IP.2)} that
\begin{align*}
\mathcal{N}^s(t)(\tau,\pi(\tau)\mathfrak{S}(\tau,x)) =&~ \pi(t) \mathcal{N}(t)(\tau,\cdot) \circ id_s \circ \pi(\tau) \circ \mathfrak{S}(\tau,x)
\\
=&~ \pi(t) \mathcal{N}(t)(\tau,\cdot) \circ id_s \circ h_u(\tau,0,x)
\\
=&~ h_u(t,0,\mathcal{N}(t)(\tau,x)),
\end{align*}
and similarly,
$$
\mathcal{N}^u(t)(\tau,(id-\pi(\tau)) \mathfrak{S}(\tau,x))= h_s(t,0,\mathcal{N}(t)(\tau,x)).
$$
Then by \eqref{SS-1}, we obtain
\begin{align*}
\mathfrak{S}(t,\mathcal{N}(t)(\tau,x)) =&~ h_s(t,0,\mathcal{N}(t)(\tau,x))+h_u(t,0,\mathcal{N}(t)(\tau,x))
\\
=&~ \mathcal{N}^u(t)(\tau,(id-\pi(\tau)) \mathfrak{S}(\tau,x)) + \mathcal{N}^s(t)(\tau,\pi(\tau)\mathfrak{S}(\tau,x)),
\end{align*}
which gives \eqref{Sp-5}. On the other hand, by also using \eqref{SS-1}, \eqref{SS-2} and {\bf (IP.1)} and {\bf (IP.2)}, we obtain
\begin{align*}
\pi(t)\mathcal{N}(t)(\tau,\hat{\mathfrak{S}}(\tau,x))
=&~ \pi(t) \mathcal{N}(t)(\tau,\cdot) \circ (\pi(\tau)\hat{\mathfrak{S}}(\tau,x)+(id-\pi(\tau))\hat{\mathfrak{S}}(\tau,x))
\\
=&~ \pi(t)\mathcal{N}(t)(\tau,\cdot)\circ ((id-\pi(\tau))\hat{\mathfrak{S}}(\tau,x)+h_u(\tau,(id-\pi(\tau))\hat{\mathfrak{S}}(\tau,x),\pi(\tau)x ))
\\
=&~ h_u(t,\pi(t)\mathcal{N}(t)(\tau,\cdot)\circ  ((id-\pi(\tau))\hat{\mathfrak{S}}(\tau,x)
\\
&\quad +h_u(\tau,(id-\pi(\tau))\hat{\mathfrak{S}}(\tau,x),\pi(\tau)x)),\mathcal{N}(t)(\tau,\pi(\tau)x))
\\
=&~ h_u(t,(id-\pi(t))\mathcal{N}(t)(\tau,\hat{\mathfrak{S}}(\tau,x)),\mathcal{N}(t)(\tau,\pi(\tau)x))
\\
=&~ h_u(t,(id-\pi(t))\mathcal{N}(t)(\tau,\hat{\mathfrak{S}}(\tau,x)),\pi(t)\mathcal{N}(t)(\tau,\cdot)\circ id_s\circ  \pi(\tau)x).
\end{align*}
Likewise, we obtain
$$
(id-\pi(t))\mathcal{N}(t)(\tau,\hat{\mathfrak{S}}(\tau,x))
= h_s(t,\pi(t)\mathcal{N}(t)(\tau,\hat{\mathfrak{S}}(\tau,x)),(id-\pi(t))\mathcal{N}(t,\cdot)\circ id_u \circ (id-\pi(t))x).
$$
Therefore,
\begin{align*}
\mathcal{N}(t)(\tau,\hat{\mathfrak{S}}(\tau,x))
=&~  \pi(t)\mathcal{N}(t)(\tau,\hat{\mathfrak{S}}(\tau,x)) + (id-\pi(t))\mathcal{N}(t)(\tau,\hat{\mathfrak{S}}(\tau,x))
\\
=&~ h_u(t,(id-\pi(t))\mathcal{N}(t)(\tau,\hat{\mathfrak{S}}(\tau,x)),\pi(t)\mathcal{N}(t)(\tau,\cdot)\circ id_s\circ  \pi(\tau)x)
\\
&~ + h_s(t,\pi(t)\mathcal{N}(t)(\tau,\hat{\mathfrak{S}}(\tau,x)),(id-\pi(t))\mathcal{N}(t,\cdot)\circ id_u \circ (id-\pi(t))x).
\end{align*}
In view of the uniqueness of solutions of \eqref{SS-2}, we assert that
$$
\mathcal{N}(t)(\tau,\hat{\mathfrak{S}}(\tau,x))
 = \hat{\mathfrak{S}}(t,\cdot)\circ  (\mathcal{N}^s(t)(\tau,\pi(\tau)x)+\mathcal{N}^u(t)(\tau,(id-\pi(\tau))x))),
$$
which proves \eqref{Sp-6}. This completes the proof of Lemma \ref{split-lem}.
\end{proof}

We proceed with the proof of Lemma \ref{lem-spli}. From  Lemma \eqref{split-lem}, it follows that
$$
\mathfrak{S}(t,\cdot)\circ \mathcal{N}(t)(\tau,\cdot)\circ \mathfrak{S}^{-1}(\tau,x)
= \mathcal{N}^s(t)(\tau,\pi(\tau)x)+\mathcal{N}^u(t)(\tau,(id-\pi(\tau))x).
$$
Noting that by {\bf (IP.1)},
$$
 \mathcal{N}^s(t)(\tau,\pi(\tau)x)
=\Psi(t,\tau)\pi(\tau)x+\int_{\tau}^t \Psi(t,\kappa)\pi(\kappa)f(\kappa,\mathcal{N}^s(\kappa)(\tau,\pi(\tau)x)) \, d\kappa,
$$
which is a solution of the first equality of \eqref{dec-eqs}, and similarly, $\mathcal{N}^u(t)(\tau,(id-\pi(\tau))x)$ is a solution of the
second equality of  \eqref{dec-eqs}. Consequently, the proof of Lemma \ref{lem-spli} is completed.
\end{proof}

\section{Proof of  Theorem \ref{thm-lin}}

Now we are ready to prove Theorem \ref{thm-lin}.

\begin{proof}[Proof of Theorem \ref{thm-lin}.]
Let $x(t):=x(t,\tau,x)$ be the solution of Eq. \eqref{nonlin-eq} with the initial value $x(\tau)=x\in\mathbb{R}^n$.
Then by Lemma \ref{lem-spli}, we understand that
$$
x(t,\tau,x)=x_s(t,\tau,x_0)+x_u(t,\tau,x_1),
$$
where $x_s(t,\tau,x_0)$ is the solution of  the first equation of \eqref{dec-eqs} with $x_s(\tau)=x_0\in \mathcal{X}_\tau^s$
and $x_u(t,\tau,x_1)$ is the solution of the second equation of \eqref{dec-eqs} with $x_u(\tau)=x_1\in \mathcal{X}_\tau^u$.
Then we have the following two lemmas.

\begin{lemma}\label{con-lin}
Suppose that $f(t,x)$ satisfies the set $\mathcal{A}_f$ with a small constant $\delta_f>0$.
Then if
$\theta\ge   \max\{\nu,\omega\}$, $\lambda_s<\nu-\theta$ and $\lambda_u>\theta-\omega$, then
there is a homeomorphism $\mathfrak{F}_s$  such that
$\mathfrak{F}_s(t,x_s(t,\tau,x_0))=\Psi_s(t,\tau)\mathfrak{F}_s(\tau,x_0)$ for all $t,\tau\in\mathbb{R}$.
\end{lemma}

\begin{lemma}\label{exp-lin}
Suppose that $f(t,x)$ satisfies the set $\mathcal{A}_f$ with a small constant $\delta_f>0$.
Then if
$\theta\ge   \max\{\nu,\omega\}$, $\lambda_s<\nu-\theta$ and $\lambda_u>\theta-\omega$, then
there is a homeomorphism $\mathfrak{F}_u$  such that
$\mathfrak{F}_u(t,x_u(t,\tau,x_1))=\Psi_u(t,\tau)\mathfrak{F}_u(\tau,x_1)$ for all $t,\tau\in\mathbb{R}$.
\end{lemma}

In what follows, we only prove Lemma \ref{con-lin} and place it at the end of this section.
The proof of Lemma  \ref{exp-lin} can be obtained similarly. We now proceed to prove Theorem  \ref{thm-lin}.
It follows from Lemmas \ref{con-lin} and \eqref{exp-lin} that there is a homeomorphism $\mathfrak{F}(\tau,\cdot):\mathbb{R}^n \to \mathbb{R}^n$ defined as
$$
\mathfrak{F}(\tau,x):=\mathfrak{F}_s(\tau,x_0)+\mathfrak{F}_u(\tau,x_1)
$$
such that
\begin{align*}
\mathfrak{F}(t,x(t,\tau,x))=&~ \mathfrak{F}_s(t,x_s(t,\tau,x_0))+\mathfrak{F}_u(t,x_u(t,\tau,x_1))
\\
=&~ \Psi_s(t,\tau)\mathfrak{F}_s(\tau,x_0)+\Psi_u(t,\tau)\mathfrak{F}_u(\tau,x_1)
\\
=&~ \Psi(t,\tau)\pi(\tau) \mathfrak{F}_s(\tau,x_0)+\Psi(t,\tau)(id-\pi(\tau))\mathfrak{F}_u(\tau,x_1)
\\
=&~ \Psi(t,\tau)\mathfrak{F}(\tau,x).
\end{align*}
Then by Lemma \ref{lem-spli}, we can define a homeomorphism $\mathfrak{G}:(\tau,\cdot):\mathbb{R}^n \to \mathbb{R}^n$ as
$$
\mathfrak{G}(\tau,x):=\mathfrak{F}(\tau,\cdot)\circ \mathfrak{S}(\tau,x)
$$
such that Eq. \eqref{nonlin-eq} and Eq. \eqref{lin-eq} are topological conjugacy. This complete the proof of Theorem \ref{thm-lin}.
\end{proof}

Finally, it remains to prove Lemma \ref{con-lin}.

\begin{proof}[Proof of Lemma \ref{con-lin}.]
Note that $\pi(t)\Psi(t,s)=\Psi(t,s)\pi(s)$. Then,  $\pi(t)$ is differentiable in $t$ and
  the chain rule applied to this yields
$$
 \pi'(t)\Psi(t,s)+\pi(t)A(t)\Psi(t,s)=A(t)\Psi(t,s)\pi(s),
$$
and setting $s=t$, we have
\begin{equation}\label{P-expr}
  \pi'(t)=A(t)\pi(t)-\pi(t)A(t).
\end{equation}
Then, we obtain from Definition \ref{quad-V}, Remark \ref{Rem-41}, \eqref{S-b}, the set $\mathcal{A}_f$ and \eqref{UUU-1} that
\begin{align*}
\frac{d}{dt} U(t,x_s(t))=&~\frac{d}{dt} U(t,\pi(t)x(t))= \frac{d}{dt} \langle S(t)\pi(t)x(t),\pi(t)x(t) \rangle
\\
=&~  \langle S'(t)\pi(t)x(t),\pi(t)x(t)\rangle+\langle S(t)(A(t)\pi(t)-\pi(t)A(t))x(t),\pi(t)x(t)\rangle
\\
&~ +\langle S(t)\pi(t)(A(t)x(t)+f(t,x(t))),\pi(t)x(t)\rangle
\\
&~ + \langle S(t)\pi(t)x(t), (A(t)\pi(t)-\pi(t)A(t))x(t) \rangle
\\
&~+  \langle S(t)\pi(t)x(t), \pi(t)(A(t)x(t)+f(t,x(t))) \rangle
\\
=&~ \langle S'(t)x_{s}(t),x_{s}(t) \rangle+\langle S(t)A(t)x_{s}(t),x_{s}(t) \rangle + \langle S(t) x_{s}(t),A(t) x_{s}(t) \rangle
\\
&~ + \langle S(t)\pi(t) f(t,x(t)),x_s(t)\rangle + \langle S(t)x_s(t), \pi(t)f(t,x(t)) \rangle
\\
=&~  \langle (S'(t)+S(t)A(t)+A(t)^*S(t))x_s(t),x_s(t) \rangle
\\
&~ +   \langle S(t)\pi(t) f(t,x(t)),x_s(t)\rangle + \langle S(t)x_s(t), \pi(t)f(t,x(t)) \rangle
\\
=&~ \langle (S'(t)+S(t)A(t)+A(t)^*S(t))x_s(t),x_s(t) \rangle
\\
&~ + \langle S(t) (\pi(t)f(t,x(t))-\pi(t)f(t,x_u(t))),x_s(t)\rangle
\\
&~ +  \langle S(t)x_s(t), \pi(t)f(t,x(t))-\pi(t)f(t,x_u(t)) \rangle
\\
\le&~ -\frac{2 \mu'(t)}{\mu(t)} \langle (id+\lambda_u S(t))x_s(t),x_s(t) \rangle
\\
&~ +  2\|S(t)\| \| \pi(t)f(t,x(t))-\pi(t)f(t,x_u(t))\| \|x_s(t)\|
\\
\le &~ -2 \left\{\frac{\mu'(t)}{\mu(t)}+\left(\lambda_u \frac{\mu'(t)}{\mu(t)}-\varphi(t) \right) \|S(t)\| \right\}\|x_s(t)\|^2
\\
\le&~ \frac{-2 \eta\left\{\frac{\mu'(t)}{\mu(t)}+\left(\lambda_u \frac{\mu'(t)}{\mu(t)}-\varphi(t) \right) \|S(t)\| \right\}}{D^2
    \max\left\{\mu(t)^{\mathrm{sign}(t)\nu},  \mu(t)^{\mathrm{sign}(t)\omega}\right\}} \, U(t,x_s(t)).
\end{align*}
For convenience, we write
$$
\mathfrak{U}(t):=\left\{\frac{\mu'(t)}{\mu(t)}+\left(\lambda_u \frac{\mu'(t)}{\mu(t)}-\varphi(t) \right) \|S(t)\| \right\}\Big/
 \max\left\{\mu(t)^{\mathrm{sign}(t)\nu},  \mu(t)^{\mathrm{sign}(t)\omega}\right\}.
$$
 Here, $\|S(t)\|$ satisfies property  \eqref{S-a}.
Clearly, $\mathfrak{U}(t)\in\mathbb{R}$ is continuous.
From the   last inequality, we affirm that
$$
U(t,x_s(t)) \le \exp\left\{\frac{-2 \eta}{D^2}\int_\tau^t\mathfrak{U}(\kappa) \, d\kappa\right\} U(\tau,x_s(\tau)), \quad \forall t\ge \tau.
$$
Then we give a qualitative analysis for  the integral
\begin{align}\label{In-U}
\frac{-2 \eta}{D^2}\int_\tau^t\mathfrak{U}(\kappa) \, d\kappa.
\end{align}
By \eqref{ff-Lip}, we have
$$
 \lambda_u \frac{\mu'(t)}{\mu(t)}-\varphi(t)= \left(\lambda_u -\delta_f \mu(t)^{-\mathrm{sign}(t)\theta} \right) \frac{\mu'(t)}{\mu(t)}.
$$
Since $\mu(t)^{-\mathrm{sign}(t)\aleph}\in (0,1], (\aleph=\nu,\omega,\theta)$ for all $t\in\mathbb{R}$ and $\delta_f$ is  small enough, we get
$\lambda_u >\delta_f \mu(t)^{-\mathrm{sign}(t)\theta}$, and thus, $\mathfrak{U}(t)>0$ because of  $\mu'(t)\ge 0$. Hence, we assert that   integral \eqref{In-U} is
strictly decreasing and
\begin{align}\label{UD}
\frac{-2 \eta}{D^2}\int_\tau^t\mathfrak{U}(\kappa) \, d\kappa =: \mathfrak{W}(t)-\mathfrak{W}(\tau)<0.
\end{align}
Recall that $U(t,x_s(t))=V^2 (t,x_s(t))$ and $V(t,x_s(t))<0$ on $\mathcal{X}_t^s\backslash\{0\}$. Then,
$$
V(t,x_s(t)) \ge V(\tau,x_s(\tau)) \exp\left\{\frac{- \eta}{D^2}\int_\tau^t\mathfrak{U}(\kappa) \, d\kappa\right\},
$$
which together with \eqref{UD} implies that $V(t,x_s(t))$ is a strictly increasing function   such that
$$
\text{ $\lim_{t\to+\infty} V(t,x_s(t))=0$ \, and \, $\lim_{t\to -\infty} V(t,x_s(t))=-\infty$.}
$$
Thus, there is a unique time $ \ell_s(\tau,x_0)\in\mathbb{R}$ satisfying
\begin{align}\label{VF-1}
V(\ell_s(\tau,x_0),x_s(\ell_s(\tau,x_0),\tau,x_0))=-1.
\end{align}
For each $x_0\ne 0$, it is easy to see that $\ell_s(\tau,x_0)$ is continuous in $(\tau,x_0)$. 

Similar to the above procedure, if $t\mapsto \Psi(t,\tau)x_0$ is the solution of Eq. \eqref{lin-eq}, then we can deduce that
$$
V(t,\Psi_s(t,\tau)x_0) \ge V(\tau,\pi(\tau)x_0) \exp\left\{\frac{- \eta}{D^2}\int_\tau^t\mathfrak{V}(\kappa) \, d\kappa\right\},
$$
where $\Psi_s(t,\tau)x_0 : =\pi(t)\Psi(t,\tau)x_0$ and
 $$
\mathfrak{V}(t):=\left\{\frac{\mu'(t)}{\mu(t)}+ \lambda_u \frac{\mu'(t)}{\mu(t)}  \|S(t)\| \right\}\Big/
 \max\left\{\mu(t)^{\mathrm{sign}(t)\nu},  \mu(t)^{\mathrm{sign}(t)\omega}\right\}.
$$
It indicates that $V(t,\Psi_s(t,\tau)x_0)$ is a strictly increasing function   such that
$$
\text{ $\lim_{t\to+\infty} V(t,\Psi_s(t,\tau)x_0)=0$ \, and \, $\lim_{t\to -\infty} V(t,\Psi_s(t,\tau)x_0)=-\infty$.}
$$
Hence, there exists  a unique time $\kappa_s(\tau,x_0)\in\mathbb{R}$ satisfying
\begin{align}\label{VF-1-1}
V(\kappa_s(\tau,x_0),\Psi_s(\kappa_s(\tau,x_0),\tau)x_0)=-1.
\end{align}
Clearly,  for $x_0\ne 0$, $\kappa_s(\tau,x_0)$ is continuous  in $(\tau,x_0)$. 

Now we use the  functions $\ell_s(\tau,x_0)$ and $\kappa_s(\tau,x_0)$ to define two mappings $\mathfrak{F}_s(\tau,x_0)$ and $\mathfrak{L}_s(\tau,x_0)$ as follows:
$$
\mathfrak{F}_s(\tau,x_0):=
\begin{cases}
x_0+\int_\tau^{\ell_s(\tau,x_0)} \Psi_s(\tau,t)\pi(t)f(t,x(t,\tau,x_0)) \, dt, \quad &x_0\ne 0,
\\
0, \quad &x_0=0,
\end{cases}
$$
and
$$
\mathfrak{L}_s(\tau,x_0):=
\begin{cases}
x_0+\int_{\kappa_s(\tau,x_0)}^\tau \Psi_s(\tau,t)\pi(t)f(t,x(t,\kappa_s(\tau,x_0),\Psi_s(\kappa_s,\tau)x_0)) \, dt, \quad &x_0\ne 0,
\\
0, \quad &x_0=0.
\end{cases}
$$
Then the Variation of Constants Formula from the first equation of  \eqref{dec-eqs} yields
\begin{align*}
\mathfrak{F}_s(\tau,x_0)=&~ x_0+\int_\tau^{\ell_s(\tau,x_0)} \Psi_s(\tau,t)\pi(t)f(t,x(t,\tau,x_0)) \, dt
\\
=&~ \Psi_s(\tau,\ell_s(\tau,x_0)) \Psi_s(\ell_s(\tau,x_0),\tau)x_0
\\
&~+ \Psi_s(\tau,\ell_s(\tau,x_0)) \int_\tau^{\ell_s(\tau,x_0)} \Psi_s(\ell_s(\tau,x_0),t)\pi(t)f(t,x(t,\tau,x_0)) \, dt
\\
=&~ \Psi_s(\tau,\ell_s(\tau,x_0)) \pi(\ell_s(\tau,x_0)) x(\ell_s(\tau,x_0),\tau,x_0)
\\
\triangleq &~ \Psi_s(\tau,\ell_s(\tau,x_0)) x_s(\ell_s(\tau,x_0),\tau,x_0), \quad \forall x_0\ne 0,
\end{align*}
and
\begin{align*}
\mathfrak{L}_s(\tau,x_0)=&~ x_0+\int_{\kappa_s(\tau,x_0)}^\tau \Psi_s(\tau,t)\pi(t)f(t,x(t,\kappa_s(\tau,x_0),\Psi_s(\kappa_s,\tau)x_0)) \, dt
\\
=&~ \Psi_s(\tau,\kappa_s(\tau,x_0)) \Psi_s(\kappa_s(\tau,x_0),\tau)x_0
\\
&~+ \int_{\kappa_s(\tau,x_0)}^\tau \Psi_s(\tau,t)\pi(t)f(t,x(t,\kappa_s(\tau,x_0),\Psi_s(\kappa_s,\tau)x_0)) \, dt
\\
=&~ \pi(\tau) x(\tau,\kappa_s(\tau,x_0),\Psi_s(\kappa_s(\tau,x_0),\tau)x_0)
\\
\triangleq&~ x_s(\tau,\kappa_s(\tau,x_0),\Psi_s(\kappa_s(\tau,x_0),\tau)x_0), \quad \forall x_0\ne 0.
\end{align*}
Therefore, we obtain that
\begin{equation}\label{TF}
\mathfrak{F}_s(\tau,x_0):=
\begin{cases}
\Psi_s(\tau,\ell_s(\tau,x_0)) x_s(\ell_s(\tau,x_0),\tau,x_0), \quad &x_0\ne 0,
\\
0, \quad &x_0=0,
\end{cases}
\end{equation}
and
\begin{equation}\label{TL}
\mathfrak{L}_s(\tau,x_0):=
\begin{cases}
x_s(\tau,\kappa_s(\tau,x_0),\Psi_s(\kappa_s(\tau,x_0),\tau)x_0), \quad &x_0\ne 0,
\\
0, \quad &x_0=0.
\end{cases}
\end{equation}

It is obvious that $\mathfrak{F}_s(\tau,x_0)$ and $\mathfrak{L}_s(\tau,x_0)$ are both continuous at $x_0\ne 0$.
We next discuss their continuity at $x_0=0$.
It follows from \eqref{VF-1},  \eqref{u-grow} and Lemma \ref{non-esti} that
\begin{align*}
-1=&~ V(\ell_s(\tau,x_0),x_s(\ell_s(\tau,x_0),\tau,x_0))
\\
\ge&   - C \, \mu(\ell_s(\tau,x_0))^{\mathrm{sign}(\ell_s(\tau,x_0))\epsilon}\|x_s(\ell_s(\tau,x_0),\tau,x_0)\|
\\
\ge&  -CD \, \mu(\ell_s(\tau,x_0))^{\mathrm{sign}(\ell_s(\tau,x_0))\epsilon} \left(\frac{\mu(\ell_s(\tau,x_0))}{\mu(\tau)}\right)^{\mathrm{sign}(\ell_s(\tau,x_0)-\tau)(\lambda_{\max}+\delta_f D)} \mu(\tau)^{\mathrm{sign}(\tau)\theta}\|x_0\|,
\end{align*}
or equivalently,
\begin{align}\label{VF-3}
\left(\frac{\mu(\ell_s(\tau,x_0))}{\mu(\tau)}\right)^{-\mathrm{sign}(\ell_s(\tau,x_0)-\tau)}
\le \left(CD \, \mu(\ell_s(\tau,x_0))^{\mathrm{sign}(\ell_s(\tau,x_0))\epsilon}  \mu(\tau)^{\mathrm{sign}(\tau)\theta}\|x_0\|    \right)^{1/(\lambda_{\max}+\delta_f D)}.
\end{align}
We then claim that
$$
\text{$x_0\in\mathbb{R}^n$, $0<\|x_0\|<  \mu(\tau)^{-\mathrm{sign}(\tau)\epsilon}/C$   $\Longrightarrow$   $\ell_s(\tau,x_0)<\tau$. }
$$
In fact, by \eqref{u-grow}, we see that for any $x_0\in\mathbb{R}^n$ with $0<\|x_0\|<  \mu(\tau)^{-\mathrm{sign}(\tau)\epsilon}/C$,
$$
V(\tau,x_0) \ge -C \mu(\tau)^{\mathrm{sign}(\tau)\epsilon}\|x_0\| > -1, 
$$
which implies that $\ell_s(\tau,x_0)<\tau$ due to $V(t,x)$ is a strictly increasing function for all $t\in\mathbb{R}$.

In view of Definition \ref{str-V}-(3), we have
\begin{align}\label{VF-4}
V(\tau,\mathfrak{F}_s(\tau,x_0)) \le - \mu(\tau)^{-\mathrm{sign}(\tau)\epsilon} \|\mathfrak{F}_s(\tau,x_0)\|/C.
\end{align}
On the other hand, by \ref{VF-1}, \ref{TF}, Definition \ref{str-V}-(2), we get
\begin{align*}
V(\tau,\mathfrak{F}_s(\tau,x_0)) =&~ V(\tau, \Psi_s(\tau,\ell_s(\tau,x_0)) x_s(\ell_s(\tau,x_0),\tau,x_0))
\\
\ge &~ V(\ell_s(\tau,x_0), \Psi_s(\tau,\ell_s(\tau,x_0)) x_s(\ell_s(\tau,x_0),\tau,x_0)) \left(\frac{\mu(\tau)}{\mu(\ell_s(\tau,x_0))}\right)^\beta
\\
\ge&~ - \left(\frac{\mu(\ell_s(\tau,x_0))}{\mu(\tau)}\right)^{-\beta},
\end{align*}
which together with \eqref{VF-4} yields
\begin{align}\label{VF-5}
\|\mathfrak{F}_s(\tau,x_0)\| \le&~ C \mu(\tau)^{\mathrm{sign}(\tau)\epsilon} \left(\frac{\mu(\ell_s(\tau,x_0))}{\mu(\tau)}\right)^{-\beta}
\notag \\
\le&~ C \mu(\tau)^{\mathrm{sign}(\tau)\epsilon}
\left(CD \, \mu(\ell_s(\tau,x_0))^{\mathrm{sign}(\ell_s(\tau,x_0))\epsilon}  \mu(\tau)^{\mathrm{sign}(\tau)\theta}\|x_0\|    \right)^{-\beta/(\lambda_{\max}+\delta_f D)}
\notag \\
\le&~ C \mu(\tau)^{\mathrm{sign}(\tau)\epsilon}
\left(BCD \,   \mu(\tau)^{\mathrm{sign}(\tau)\theta}\|x_0\|    \right)^{-\beta/(\lambda_{\max}+\delta_f D)},
\end{align}
where $\mu(\ell_s(\tau,x_0))^{\mathrm{sign}(\ell_s(\tau,x_0))\epsilon} \le B$ for
any given $\tau\in\mathbb{R}$ and some constant $B>0$. 
The last inequality implies that $\|\mathfrak{F}_s(\tau,x_0)\| \to 0$ as $\|x_0\| \to 0$ for each fixed  $\tau\in\mathbb{R}$,
namely,  $\mathfrak{F}_s(\tau,x_0)$ is continuous  at $x_0=0$.

On can check that $\mathfrak{L}_s(\tau,x_0)$ is continuous at $x_0=0$ in an analogous way. Actually, similar to \eqref{VF-3}, we have
\begin{align}\label{VF-6}
\left(\frac{\mu(\kappa_s(\tau,x_0))}{\mu(\tau)}\right)^{-\mathrm{sign}(\kappa_s(\tau,x_0)-\tau)}
\le \left(CD \, \mu(\kappa_s(\tau,x_0))^{\mathrm{sign}(\kappa_s(\tau,x_0))\epsilon}  \mu(\tau)^{\mathrm{sign}(\tau)\theta}\|x_0\|    \right)^{1/\lambda_{\max}},
\end{align}
and $\kappa_s(\tau,x_0)<\tau$ if $0<\|x_0\| < \mu(\tau)^{-\mathrm{sign}(\kappa_s(\tau,x_0))\epsilon}/C$. Similar to \eqref{VF-5}, but using \eqref{VF-1-1},  \eqref{TL} and
\eqref{VF-6} instead of \eqref{VF-1}, \eqref{TF} and \eqref{VF-3}, we obtain
\begin{align*}
\|\mathfrak{L}_s(\tau,x_0)\| \le&~ C \mu(\tau)^{\mathrm{sign}(\tau)\epsilon} \left(\frac{\mu(\kappa_s(\tau,x_0))}{\mu(\tau)}\right)^{-\beta}
\\
\le &~ C \mu(\tau)^{\mathrm{sign}(\tau)\epsilon}
\left(CD \, \mu(\kappa_s(\tau,x_0))^{\mathrm{sign}(\kappa_s(\tau,x_0))\epsilon}  \mu(\tau)^{\mathrm{sign}(\tau)\theta}\|x_0\|    \right)^{-\beta/\lambda_{\max}}
\\
\le&~ C \mu(\tau)^{\mathrm{sign}(\tau)\epsilon} \left(\tilde{B}CD \,   \mu(\tau)^{\mathrm{sign}(\tau)\theta}\|x_0\|    \right)^{-\beta/\lambda_{\max}},
\end{align*}
which indicates that $\mathfrak{L}_s(\tau,x_0)$ is continuous at $x_0=0$ for any fixed $\tau\in\mathbb{R}$.

Next, we claim that
\begin{align}\label{FL-T}
\text{ $\mathfrak{F}_s(t,x_s(t,\tau,x_0))=\Psi_s(t,\tau) \mathfrak{F}_s(\tau,x_0)$ \, and \, $\mathfrak{L}_s(t,\Psi_s(t,\tau)x_0)=x_s(t,\tau,\mathfrak{L}_s(\tau,x_0))$ }
\end{align}
for all $t,\tau \in\mathbb{R}$ and all $x_0\in\mathbb{R}^n$.
In fact, if $x_0=0$, then  $\mathfrak{F}_s(t,x_s(t,\tau,x_0))=\mathfrak{F}_s(t,0)=0$.
In the case of $x_0\ne 0$, by \eqref{TF} we affirm that
\begin{align}\label{VF-8}
\mathfrak{F}_s(t,x_s(t,\tau,x_0))=&~ \Psi_s(t,\ell_s(t,x_s(t,\tau,x_0))) x_s(\ell_s(t,x_s(t,\tau,x_0)),t,x_s(t,\tau,x_0))
\notag \\
=&~ \Psi_s(t,\ell_s(t,x_s(t,\tau,x_0))) x_s(\ell_s(t,x_s(t,\tau,x_0)),\tau,x_0).
\end{align}
From \eqref{VF-1}, it follows that
\begin{align*}
-1=&~ V(\ell_s(\tau,x_0),x_s(\ell_s(\tau,x_0),\tau,x_0))
\\
=&~ V(\ell_s(\tau,x_0), x_s(\ell_s(\tau,x_0),t,x_s(t,\tau,x_0)))
\end{align*}
and
\begin{align*}
-1= V(\ell_s(t,x_s(t,\tau,x_0)),x_s(\ell_s(t,x_s(t,\tau,x_0)),t,x_s(t,\tau,x_0))).
\end{align*}
Recall that $V(t,x)$ is a  strict increasing function with respect to $t$. Then,
$$
\ell_s(t,x_s(t,\tau,x_0))=\ell_s(\tau,x_0),
$$
which together with \eqref{VF-8} and \eqref{TF} yields
\begin{align*}
\mathfrak{F}_s(t,x_s(t,\tau,x_0))=&~ \Psi_s(t,\ell_s(\tau,x_0)) x_s(\ell_s(\tau,x_0),\tau,x_0)
\\
=&~ \Psi_s(t,\tau) \Psi_s(\tau,\ell_s(\tau,x_0)) x_s(\ell_s(\tau,x_0),\tau,x_0)
\\
=&~ \Psi_s(t,\tau) \mathfrak{F}_s(\tau,x_0).
\end{align*}
This proves the first equality of \eqref{VF-8}.
Similarly, one can also obtain that $\mathfrak{L}_s(t,\Psi_s(t,\tau)x_0)=x_s(t,\tau,\mathfrak{L}_s(\tau,x_0))$.   

We then claim that for each $\tau\in\mathbb{R}$ and  $x_0\in\mathbb{R}^n$,
\begin{align}\label{VL-0}
\text{$\mathfrak{F}_s(\tau,\mathfrak{L}_s(\tau,x_0))=x_0$ \, and \, $\mathfrak{L}_s(\tau,\mathfrak{F}_s(\tau,x_0))=x_0$. }
\end{align}
In view of \eqref{TL},
\begin{align}\label{VF-9}
x_s(\kappa_s(\tau,x_0),\tau,\mathfrak{L}_s(\tau,x_0))=&~ x_s(\kappa_s(\tau,x_0),\tau, x_s(\tau,\kappa_s(\tau,x_0),\Psi_s(\kappa_s(\tau,x_0),\tau)x_0))
\notag \\
=&~ \Psi_s(\kappa_s(\tau,x_0),\tau)x_0,
\end{align}
and by \eqref{VF-1-1},  we obtain
\begin{align}\label{VL-1}
-1=&~ V(\kappa_s(\tau,x_0),\Psi_s(\kappa_s(\tau,x_0),\tau)x_0)
\notag \\
=&~ V(\kappa_s(\tau,x_0),x_s(\kappa_s(\tau,x_0),\tau,\mathfrak{L}_s(\tau,x_0))).
\end{align}
On the other hand, by \eqref{VF-1}   we have
\begin{align}\label{VL-3}
-1=V(\ell_s(\tau,\mathfrak{L}_s(\tau,x_0)),x_s(\ell_s(\tau,\mathfrak{L}_s(\tau,x_0)),\tau,\mathfrak{L}_s(\tau,x_0))).
\end{align}
Combining \eqref{VL-1} and \eqref{VL-3}, we assert that
\begin{align}\label{VL-4}
\kappa_s(\tau,x_0)=\ell_s(\tau,\mathfrak{L}_s(\tau,x_0)).
\end{align}
It follows from \eqref{TF}, \eqref{VF-9} and \eqref{VL-4} that
\begin{align*}
\mathfrak{F}_s(\tau,\mathfrak{L}_s(\tau,x_0))=&~ \Psi_s(\tau,\ell_s(\tau,\mathfrak{L}_s(\tau,x_0)))x_s(\ell_s(\tau,\mathfrak{L}_s(\tau,x_0)),\tau,\mathfrak{L}_s(\tau,x_0))
\\
=&~ \Psi_s(\tau,\kappa_s(\tau,x_0)) x_s(\kappa_s(\tau,x_0),\tau,\mathfrak{L}_s(\tau,x_0))
\\
=&~ \Psi_s(\tau,\kappa_s(\tau,x_0))\Psi_s(\kappa_s(\tau,x_0),\tau)x_0=x_0,
\end{align*}
which gives the first equality of \eqref{VL-0}.
Similarly, one can verify that  $\ell_s(\tau,x_0)=\kappa_s(\tau,\mathfrak{F}_s(\tau,x_0))$ and  $\mathfrak{L}_s(\tau,\mathfrak{F}_s(\tau,x_0))=x_0$.
This completes the proof.
\end{proof}

\appendix

\section{Proof of Lemma \ref{non-esti}}\label{App-2}

\begin{proof}[Proof of Lemma \ref{non-esti}.]
The Variation of the Constant Formula tells us that
$$
x(t,\tau,x_0)=\Psi(t,\tau)x_0 + \int_\tau^t \Psi(t,\kappa) f(\kappa,x(\kappa,\tau,x_0)) \, d\kappa
$$
for all $t,\tau\in\mathbb{R}$ and $x_0  \in\mathbb{R}^n$.
In what follows, we only prove the case of \(t \ge \tau\),  regarding  \(t < \tau\) one can be obtained in a similar way.
It follows from \eqref{u-bogr} and the definition of  $\mathcal{A}_f$ that for any $t\ge\tau$,
\begin{align*}
\|x(t,\tau,x_0)-x(t,\tau,\tilde{x}_0)\| \le&~ D \left(\frac{\mu(t)}{\mu(\tau)}\right)^{\lambda_{\max}} \mu(\tau)^{\mathrm{sign}(\tau)\theta} \|x_0-\tilde{x}_0\|
\\
&~ + \int_\tau^t D \left(\frac{\mu(t)}{\mu(\kappa)}\right)^{\lambda_{\max}} \mu(\kappa)^{\mathrm{sign}(\kappa)\theta}\cdot \varphi(\kappa) \|x(\kappa,\tau,x_0)-x(\kappa,\tau,\tilde{x}_0)\| \, d\kappa,
\end{align*}
or equivalently,
\begin{align*}
& \left(\frac{\mu(t)}{\mu(\tau)}\right)^{-\lambda_{\max}} \mu(\tau)^{-\mathrm{sign}(\tau)\theta}\|x(t,\tau,x_0)-x(t,\tau,\tilde{x}_0)\|
\\
&~ \le D \|x_0-\tilde{x}_0\| +  D\int_\tau^t\mu(\kappa)^{\mathrm{sign}(\kappa)\theta} \varphi(\kappa)\cdot
 \left(\frac{\mu(\kappa)}{\mu(\tau)}\right)^{-\lambda_{\max}} \mu(\tau)^{-\mathrm{sign}(\tau)\theta}\|x(\kappa,\tau,x_0)-x(\kappa,\tau,\tilde{x}_0)\| \, d\kappa.
\end{align*}
By applying Gronwall's inequality, we see that
\begin{align}\label{xxx-est}
\left(\frac{\mu(t)}{\mu(\tau)}\right)^{-\lambda_{\max}} \mu(\tau)^{-\mathrm{sign}(\tau)\theta}\|x(t,\tau,x_0)-x(t,\tau,\tilde{x}_0)\|
\le D e^{D \int_\tau^t \varphi(\kappa) \mu(\kappa)^{\mathrm{sign}(\kappa)\theta} \, d\kappa}\|x_0-\tilde{x}_0\|,
  \quad \forall t\ge \tau.
\end{align}
In view of \eqref{ff-Lip}, we get
\begin{align*}
  \int_\tau^t \varphi(\kappa) \mu(\kappa)^{\mathrm{sign}(\kappa)\theta} \, d\kappa
  =&~  \int_\tau^t \delta_f \mu(\kappa)^{-1-\mathrm{sign}(\kappa)\theta} \mu'(\kappa)   \mu(\kappa)^{\mathrm{sign}(\kappa)\theta} \, d\kappa
\\
=&~ \delta_f \int_\tau^t \mu(\kappa)^{-1} d\mu(\kappa) = \delta_f \log\left(\frac{\mu(t)}{\mu(\tau)}\right),
\end{align*}
which together with \eqref{xxx-est} yields
$$
\|x(t,\tau,x_0)-x(t,\tau,\tilde{x}_0)\| \le  D
  \left(\frac{\mu(t)}{\mu(\tau)}\right)^{ \lambda_{\max}+D\delta_f}\mu(\tau)^{\mathrm{sign}(\tau)\theta}\|x_0-\tilde{x}_0\|,
  \quad \forall t\ge \tau.
$$
Replacing $x_0$ and $\tilde{x}_0$ with $x(\tau,t,x_0)$ and $x(\tau,t,\tilde{x}_0)$ respectively, it follows that
$$
\|x(\tau,t,x_0)-x(\tau,t,\tilde{x}_0)\| \ge
\frac{1}{D} \left(\frac{\mu(t)}{\mu(\tau)}\right)^{- (\lambda_{\max}+D\delta_f)}\mu(\tau)^{-\mathrm{sign}(\tau)\theta}
\|x_0-\tilde{x}_0\|.
$$
This completes the proof.
\end{proof}

\section{Proof of  Lemmas  \ref{thm-stab} and  \ref{thm-foli}}\label{App-3}

\begin{proof}[Proof of Lemma \ref{thm-stab}.]
We employ the classical Lyapunov-Perron method to prove the existence of the stable manifold (see \cite{BSV-JLMS,BS-JFA,JDDE-BS,ZWN-AM}).
The objective is to investigate the following Lyapunov-Perron equation:
\begin{align}\label{LP-eqs}
x(t)=&~ \Psi(t,\tau)\pi(\tau)x_0 + \int_\tau^t \Psi(t,\kappa)\pi(\kappa) f(\kappa,x(\kappa,\tau,x_0)) \, d\kappa
\notag \\
&~ -\int_t^{+\infty} \Psi(t,\kappa)(id-\pi(\kappa)) f(\kappa,x(\kappa,\tau,x_0)) \, d\kappa, \quad \forall t\ge \tau,
\end{align}

Let $\mathcal{H}$ denote the space of all continuous mapping $g:[\tau,+\infty) \to \mathbb{R}^n$ such that
$\|g\|_{\mathcal{H}}:=\sup_{t\ge \tau}   \|g(t)\|$.
Then $(\mathcal{H},\|\cdot\|_{\mathcal{H}})$ is a Banach space. Define an operator $\mathcal{T}:\mathcal{H}\to \mathcal{H}$ by
\begin{align*}
\mathcal{T}x(t)=&~ \Psi(t,\tau)\pi(\tau)x_0 + \int_\tau^t \Psi(t,\kappa)\pi(\kappa) f(\kappa,x(\kappa,\tau,x_0)) \, d\kappa
\\
&~ -\int_t^{+\infty} \Psi(t,\kappa)(id-\pi(\kappa)) f(\kappa,x(\kappa,\tau,x_0)) \, d\kappa, \quad \forall t\ge \tau,
\end{align*}
We claim that $\mathcal{T}x\in \mathcal{H}$ for all $x\in \mathcal{H}$.
In fact,
noting that $\theta\ge   \max\{\nu,\omega\}$, $\lambda_s<\nu-\theta$ and $\lambda_u>\theta-\omega$,  we have
$$
\text{$\mu(t)^{-\mathrm{sign}(t)(\theta- \nu)}\in (0,1]$ \, and \, $\mu(t)^{-\mathrm{sign}(t)(\theta- \omega)}\in (0,1]$, \, $\forall t\in\mathbb{R}$, }
$$
and for $t\ge \tau$,
\begin{align*}
& \int_\tau^t    \left(\frac{\mu(t)}{\mu(\kappa)} \right)^{\lambda_s} \mu(\kappa)^{\mathrm{sign}(\kappa)( \nu-\theta)-1}\mu'(\kappa)    \, d\kappa
\\
& = \mu(t)^{\lambda_s} \int_\tau^t    \mu(\kappa)^{\mathrm{sign}(\kappa)( \nu-\theta)-1-\lambda_s} \, d\mu(\kappa)
\\
&= \frac{\mu(t)^{-\mathrm{sign}(t)(\theta- \nu)}}{-\lambda_s + \mathrm{sign}(t)( \nu-\theta)}
 - \frac{ \mu(t)^{\lambda_s} \mu(\tau)^{-\lambda_s}\mu(\tau)^{-\mathrm{sign}(\tau)(\theta- \nu)}}{-\lambda_s + \mathrm{sign}(\tau)( \nu-\theta)}
\\
&\le \frac{1}{-\lambda_s +   \nu-\theta},
\end{align*}
and for $t\le \tau$,
\begin{align*}
& \int_t^{+\infty}   \left(\frac{\mu(t)}{\mu(\kappa)} \right)^{\lambda_u}
            \mu(\kappa)^{\mathrm{sign}(\kappa)( \omega-\theta)-1}  \mu'(\kappa)    \, d\kappa
\\
&=  \mu(t)^{\lambda_u} \int_t^{+\infty} \mu(\kappa)^{\mathrm{sign}(\kappa)( \omega-\theta)-1-\lambda_u} \, d\mu(\kappa)
\\
&= \frac{\mu(t)^{-\mathrm{sign}(t)(\theta -\omega)}}{\lambda_u-\mathrm{sign}(t)( \omega-\theta)}
\\
& \le \frac{1}{\lambda_u-\theta +\omega}.
\end{align*}
Then, for any given $\tau\in\mathbb{R}$, when $t\ge \tau$, it follows from \eqref{u-dich} and the definition of set $\mathcal{A}_f$ that
\begin{align*}
\|\mathcal{T}x\|_{\mathcal{H}}
\le&~    D \left(\frac{\mu(t)}{\mu(\tau)} \right)^{\lambda_s} \mu(\tau)^{\mathrm{sign}(\tau)\nu}\|x_0\|
+ \int_\tau^t D  \left(\frac{\mu(t)}{\mu(\kappa)} \right)^{\lambda_s} \mu(\kappa)^{\mathrm{sign}(\kappa)\nu} \cdot \varphi(\kappa) \|x(\kappa,\tau,x_0)\| \, d\kappa
\\
& + \int_t^{+\infty} D \left(\frac{\mu(t)}{\mu(\kappa)} \right)^{\lambda_u} \mu(\kappa)^{\mathrm{sign}(\kappa)\omega}
    \cdot \varphi(\kappa) \|x(\kappa,\tau,x_0)\| \, d\kappa
\\
\le&~ D \left(\frac{\mu(t)}{\mu(\tau)} \right)^{\lambda_s} \|x_0\|
+ \int_\tau^t D\delta_f \left(\frac{\mu(t)}{\mu(\kappa)} \right)^{\lambda_s} \mu(\kappa)^{\mathrm{sign}(\kappa)( \nu-\theta)-1}\mu'(\kappa)  \|x\|_{\mathcal{H}} \, d\kappa
\\
&  + \int_t^{+\infty} D\delta_f \left(\frac{\mu(t)}{\mu(\kappa)} \right)^{\lambda_u}
            \mu(\kappa)^{\mathrm{sign}(\kappa)( \omega-\theta)-1}  \mu'(\kappa)  \|x\|_{\mathcal{H}} \, d\kappa
\\
\le&~  D \left(\frac{\mu(t)}{\mu(\tau)} \right)^{\lambda_s}  \mu(\tau)^{\mathrm{sign}(\tau)\nu} \|x_0\|
  + \delta_f D \left( \frac{1}{-\lambda_s +   \nu-\theta}+ \frac{1}{\lambda_u-\theta +\omega}\right) \|x\|_{\mathcal{H}}<+\infty.
\end{align*}
Moreover, for any $x,\tilde{x}\in\mathcal{H}$, we get
\begin{align*}
\|\mathcal{T}x-\mathcal{T}\tilde{x}\|_{\mathcal{H}}
 \le&~ \int_\tau^t D  \left(\frac{\mu(t)}{\mu(\kappa)} \right)^{\lambda_s} \mu(\kappa)^{\mathrm{sign}(\kappa)\nu} \cdot \varphi(\kappa) \|x-\tilde{x}\|_{\mathcal{H}} \, d\kappa
\\
&~+ \int_t^{+\infty} D \left(\frac{\mu(t)}{\mu(\kappa)} \right)^{\lambda_u} \mu(\kappa)^{\mathrm{sign}(\kappa)\omega}
    \cdot \varphi(\kappa) \|x-\tilde{x}\|_{\mathcal{H}} \, d\kappa
\\
\le&~  \delta_f D \left( \frac{1}{-\lambda_s +   \nu-\theta}+ \frac{1}{\lambda_u-\theta +\omega}\right) \|x-\tilde{x}\|_{\mathcal{H}}.
\end{align*}
Thus the operator $\mathcal{T}$ is a contraction on $\mathcal{H}$ since one can take $\delta_f$ sufficiently small such that
$$
\tilde{\delta}_f:=\delta_f D \left( \frac{1}{-\lambda_s +   \nu-\theta}+ \frac{1}{\lambda_u-\theta +\omega}\right) \in (0,1).
$$
It means that $\mathcal{T}$ has a unique fixed point $x_t(\tau,\xi):=x(t,\tau,\xi)$ with $\xi=\pi(\tau)x_0\in \mathcal{X}_\tau^s$.

For any given $\tau\in\mathbb{R}$ and any $\xi,\tilde{\xi}\in \mathcal{X}_\tau^s$, proceeding as above, we get
\begin{align*}
 \|x_t(\tau,\xi)-x_t(\tau,\tilde{\xi})\|
  \le
D \left(\frac{\mu(t)}{\mu(\tau)} \right)^{\lambda_s} \mu(\tau)^{\mathrm{sign}(\tau)\nu} \|\xi-\tilde{\xi}\|
+\tilde{\delta}_f   \, \|x_t(\tau,\xi)-x_t(\tau,\tilde{\xi})\|
\end{align*}
and thus,
\begin{align}\label{Lip-Man}
 \|x_t(\tau,\xi)-x_t(\tau,\tilde{\xi})\|
\le \frac{D}{1-\tilde{\delta}_f} \left(\frac{\mu(t)}{\mu(\tau)} \right)^{\lambda_s}  \mu(\tau)^{\mathrm{sign}(\tau)\nu} \|\xi-\tilde{\xi}\|,
\end{align}
which proves the Lipschitz continuity of $x_t(\tau,\xi)$ in $\xi$.
By the uniqueness of solution of Eq. \eqref{nonlin-eq}, we have that $x_t(\tau,0)=0$.

Define
$$
\text{$\mathcal{G}_s:=\big\{(\tau,y(\tau))\in\mathbb{R}\times \mathbb{R}^n: y(t,\tau,y(\tau))$ is defined in the set $\mathcal{H} \big\}$.  }
$$
Then we obtain from \eqref{LP-eqs} and the contraction of $\mathcal{T}$ that the initial value $y(\tau)$, which compose the set $\mathcal{G}_s$
can be expressed as:
$$
y(\tau)=x_\tau(\tau,\xi)=\xi+\int_\tau^{+\infty} \Psi(\tau,\kappa)(id-\pi(\kappa)) f(\kappa,x(\kappa,\tau,x_0)) \, d\kappa
=: \xi + g_s(\tau,\xi).
$$
By \eqref{Lip-Man}, we have
\begin{align*}
\|g_s(\tau,\xi)-g_s(\tau,\tilde{\xi})\|
\le&~ \int_\tau^{+\infty}
  D \left(\frac{\mu(\tau)}{\mu(\kappa)} \right)^{\lambda_u} \mu(\kappa)^{\mathrm{sign}(\kappa)\omega}
 \varphi(\kappa) \|x_{\kappa}(\tau,\xi)-x_{\kappa}(\tau,\tilde{\xi})\| \, d\kappa
\\
\le&~ \frac{\delta_f D^2}{1-\tilde{\delta}_f} \mu(\tau)^{\mathrm{sign}(\tau)\nu}\|\xi-\tilde{\xi}\|
 \int_\tau^{+\infty} \left(\frac{\mu(\tau)}{\mu(\kappa)} \right)^{\lambda_u} \mu(\kappa)^{\mathrm{sign}(\kappa)(\omega-\theta)-1}\mu'(\kappa) \, d\kappa
\\
\le&~  \frac{\delta_f D^2}{1-\tilde{\delta}_f} \mu(\tau)^{\mathrm{sign}(\tau)\nu}\|\xi-\tilde{\xi}\| \frac{1}{\lambda_u-\theta+\omega}
=: \mathrm{Lip}(g_s) \mu(\tau)^{\mathrm{sign}(\tau)\nu} \|\xi-\tilde{\xi}\|,
\end{align*}
which show that $g_s(\tau,\xi)$ is Lipschitz continuous in $\xi$ for any given $\tau\in\mathbb{R}$.
Furthermore, $g_s(\tau,0)=0$ due to $x_\tau(\tau,0)=0$ and $f(\tau,0)=0$.
Hence, $\mathcal{G}_s:=\{\xi+g_s(\tau,\xi): \tau\in\mathbb{R},\xi\in\mathcal{X}_\tau^s\}$ is a Lipschitz stable manifold.
The invariance of  $\mathcal{G}_s$ is similar to \cite[Theorem 3.1]{ZWN-AM}, so we omit its proof.
Therefore,   the proof of Lemma \ref{thm-stab} is completed.
\end{proof}

\begin{proof}[Proof of Lemma \ref{thm-foli}.]
Given $\tau\in\mathbb{R}$, for all $t\ge \tau$, we consider the set
$$
\mathcal{W}_s(\tau,x)=\big\{y\in\mathbb{R}^n: x(t,\tau,x)-x(t,\tau,y) \in \mathcal{H} \big\}.
$$
Letting $p(t)=x(t,\tau,x)-x(t,\tau,y)$,  by employing a strategy similar to that in \cite[Theorem 2.1]{CHT-JDE}, we can understand that
$y\in\mathcal{W}_s(\tau,x)$ if and only if $p(t)\in\mathcal{H}$ and satisfies the following Lyapunov-Perron formula
\begin{align}\label{LP-foi}
p(t)
=&~ \Psi(t,\tau)\pi(\tau)\eta + \int_\tau^t \Psi(t,\kappa) \pi(\kappa)
               \big\{ f(\kappa,p(\kappa)+x(\kappa,\tau,x))-f(\kappa,x(\kappa,\tau,x)) \big\} \, d\kappa
\notag \\
&~ - \int_t^{+\infty} \Psi(t,\kappa)(id-\pi(\kappa)) \big\{ f(\kappa,p(\kappa)+x(\kappa,\tau,x))-f(\kappa,x(\kappa,\tau,x)) \big\} \, d\kappa,
\end{align}
where $\eta=\pi(\tau)(x-y)$.
In order to derive the existence of the stable foliation, for each $p(t)\in\mathcal{H}$,
we define an operator $\mathcal{J}:\mathcal{H}\to \mathcal{H}$ by
\begin{align*}
(\mathcal{J}p)(t)
=&~ \Psi(t,\tau)\pi(\tau)\eta + \int_\tau^t \Psi(t,\kappa) \pi(\kappa)
               \big\{ f(\kappa,p(\kappa)+x(\kappa,\tau,x))-f(\kappa,x(\kappa,\tau,x)) \big\} \, d\kappa
\\
&~ - \int_t^{+\infty} \Psi(t,\kappa)(id-\pi(\kappa)) \big\{ f(\kappa,p(\kappa)+x(\kappa,\tau,x))-f(\kappa,x(\kappa,\tau,x)) \big\} \, d\kappa.
\end{align*}
Then, similar to the proof of Lemma \ref{thm-stab}, one can   verify that
$$
\|\mathcal{J}p\|_{\mathcal{H}} \le  D\left(\frac{\mu(t)}{\mu(\tau)} \right)^{\lambda_s}  \mu(\tau)^{\mathrm{sign}(\tau)\nu} \|\eta\|
  +\tilde{\delta}_f \|p\|_{\mathcal{H}}<+\infty
$$
and
$$
\|\mathcal{J}p-\mathcal{J}\tilde{p}\|_{\mathcal{H}} \le \tilde{\delta}_f \|p-\tilde{p}\|_{\mathcal{H}}.
$$
Thus, for each $\eta\in \mathcal{X}_\tau^s$ and $x\in\mathbb{R}^n$,
$\mathcal{J}$ is a contraction mapping and has a unique fixed point $p_t(\tau,\eta,x):=p(t,\tau,\eta,x)\in \mathcal{H}$.
Moreover,  for $\eta,\tilde{\eta}\in \mathcal{X}_\tau^s$, we have
\begin{align}\label{pt-est}
\|p_t(\tau,\eta,x)-p_t(\tau,\tilde{\eta},x)\| \le \frac{D}{1-\tilde{\delta}_f}  \left(\frac{\mu(t)}{\mu(\tau)} \right)^{\lambda_s}  \mu(\tau)^{\mathrm{sign}(\tau)\nu} \|\eta-\tilde{\eta}\|.
\end{align}
  Letting $\zeta=\pi(\tau)(p_\tau(\tau,\eta,x)+x)$, then we define $h_s(\tau,\zeta,x):=(id-\pi(\tau))(p_\tau(\tau,\eta,x)+x)$, i.e.,
\begin{align*}
&h_s(\tau,\zeta,x)
\\
&~ =(id-\pi(\tau))x - \int_{\tau}^{+\infty} \Psi(\tau,\kappa)  (id-\pi(\kappa)) \big\{ f(\kappa,p(\kappa)+x(\kappa,\tau,x))-f(\kappa,x(\kappa,\tau,x)) \big\} \, d\kappa.
\end{align*}
Since $\eta=\pi(\tau)p_\tau(\tau,\eta,x)$, we have $\zeta=\eta+\pi(\tau)x$.
For any $\zeta,\tilde{\zeta}\in \mathcal{X}_{\tau}^s$ and $x\in\mathbb{R}^n$, by \eqref{pt-est} we get
\begin{align*}
&\|h_s(\tau,\zeta,x)-h_s(\tau,\tilde{\zeta},x)\|
\\
&~ \le  \int_{\tau}^{+\infty} D \left(\frac{\mu(\tau)}{\mu(\kappa)} \right)^{\lambda_u} \mu(\kappa)^{\mathrm{sign}(\kappa)\omega}
 \varphi(\kappa) \|p_{\kappa}(\tau,\zeta-\pi(\tau)x,x)-p_{\kappa}(\tau,\tilde{\zeta}-\pi(\tau)x,x)\| \, d\kappa
\\
&~ \le  \frac{\delta_f D^2}{(1-\tilde{\delta}_f)(\lambda_u-\theta+\omega)} \mu(\tau)^{\mathrm{sign}(\tau)\nu}\|\zeta-\tilde{\zeta}\|
:= \mathrm{Lip}(h_s) \mu(\tau)^{\mathrm{sign}(\tau)\nu}\|\zeta-\tilde{\zeta}\|,
\end{align*}
which implies that $h_s(\tau,\zeta,x)$ is Lipschitz continuous in $\zeta$ for any given $\tau\in\mathbb{R}$.

Next, we claim that $h_s(\tau,\zeta,x)$ is continuous in $x$. In fact, it is equivalent to proving that $p_t(\tau,\eta,x)$ is continuous in $x$.
By \eqref{u-dich} and the definition of  $\mathcal{A}_f$, we have
\begin{small}
\begin{align*}
&\sup_{t\ge \tau} \|p_t(\tau,\eta,x)-p_t(\tau,\eta,\tilde{x})\|
\\
&~\le \sup_{t\ge \tau}\Bigg\{ \int_\tau^t D  \left(\frac{\mu(t)}{\mu(\kappa)} \right)^{\lambda_s} \mu(\kappa)^{\mathrm{sign}(\kappa)\nu} \varphi(\kappa)
\big(\|p_{\kappa}(\tau,\eta,x)-p_{\kappa}(\tau,\eta,\tilde{x})\| +2\|x(\kappa,\tau,x)-x(\kappa,\tau,\tilde{x})\| \big) \, d\kappa
\\
&\quad  + \int_t^{+\infty} D \left(\frac{\mu(t)}{\mu(\kappa)} \right)^{\lambda_u} \mu(\kappa)^{\mathrm{sign}(\kappa)\omega}
 \varphi(\kappa) \big(\|p_{\kappa}(\tau,\eta,x)-p_{\kappa}(\tau,\eta,\tilde{x})\| +2\|x(\kappa,\tau,x)-x(\kappa,\tau,\tilde{x})\| \big) \, d\kappa \Bigg\}
\\
&~\le \tilde{\delta}_f  \sup_{t\ge \tau} \|p_t(\tau,\eta,x)-p_t(\tau,\eta,\tilde{x})\|
  + 2\tilde{\delta}_f \sup_{t\ge \tau} \|x(t,\tau,x)-x(t,\tau,\tilde{x})\|
\end{align*}
\end{small}
and thus,
\begin{align}\label{pt-con}
\sup_{t\ge \tau} \|p_t(\tau,\eta,x)-p_t(\tau,\eta,\tilde{x})\|  \le \frac{2\tilde{\delta}_f}{1-\tilde{\delta}_f}
           \sup_{t\ge \tau} \|x(t,\tau,x)-x(t,\tau,\tilde{x})\|.
\end{align}
Then in order to prove the continuity of $p_t(\tau,\eta,x)$ in $x$, we choose a $\varrho\in(1,-\lambda_s)$ and let
$\mathcal{H}_{\varrho}$ be the space of all continuous mapping $g:[\tau,+\infty)\to \mathbb{R}^n$ such that
$\|g\|_{\mathcal{H}_{\varrho}}:=\sup_{t\ge \tau} \left( \frac{\mu(t)}{\mu(\tau)}\right)^{\varrho}\|g(t)\|$.
Then, $(\mathcal{H}_{\varrho},\|\cdot\|_{\varrho})$ is a Banach space.
Similarly, we can verify that $\mathcal{J}$ also has a unique fixed point in $\mathcal{H}_{\varrho}$.
In addition, this fixed point is also a solution of Eq. \eqref{LP-foi} in $\mathcal{H}$  because of
$(\mathcal{H}_{\varrho},\|\cdot\|_{\mathcal{H}_{\varrho}})\subset (\mathcal{H},\|\cdot\|_{\mathcal{H}})$.
The uniqueness asserts that  this fixed point is just $p_t(\tau,\eta,x)$.
Therefore, there is a bound $M>0$ such that $\|p(\tau,\eta,x)\|_{\mathcal{H}_{\varrho}}\le M$ and for $t\ge \varsigma \ge \tau$, we see that
$$
\sup_{t\ge \varsigma} \|p_t(\tau,\eta,x)\|
 \le \|p(\tau,\eta,x)\|_{\mathcal{H}_{\varrho}} \sup_{t\ge \varsigma} \left(\frac{\mu(t)}{\mu(\tau)}\right)^{-\varrho}
 \le M  \left(\frac{\mu(\tau)}{\mu(t)}\right)^{\varrho} \to 0,
$$
as $\varsigma\to +\infty$. For any given $\varepsilon>0$, we can choose a $\varsigma$ satisfying
$\mu(\varsigma)/\mu(\tau)>(4M/\varepsilon)^{1/\varrho}$.
Then in the case of $t\ge \varsigma$, we get
$$
\sup_{t\ge \varsigma} \|p_t(\tau,\eta,x)-p_t(\tau,\eta,\tilde{x})\| \le 2M \left(\frac{\mu(\tau)}{\mu(t)}\right)^{\varrho} < \varepsilon/2.
$$
In the case of $\tau\le t <\varsigma$, similarly to \eqref{pt-con}, we have
$$
\sup_{t\in [\tau,\varsigma)}  \|p_t(\tau,\eta,x)-p_t(\tau,\eta,\tilde{x})\|
 \le \frac{2\tilde{\delta}_f}{1-\tilde{\delta}_f}  \sup_{t\in [\tau,\varsigma)} \|x(t,\tau,x)-x(t,\tau,\tilde{x})\| < \varepsilon/2,
$$
as $\|x-\tilde{x}\|\to 0$ due to  the continuous dependence on the initial value and the compactness of $[\tau,\varsigma)$.
Consequently,
\begin{align*}
&\sup_{t\ge \tau} \|p_t(\tau,\eta,x)-p_t(\tau,\eta,\tilde{x})\|
\\
&~ \le
\sup_{t\ge \varsigma} \|p_t(\tau,\eta,x)-p_t(\tau,\eta,\tilde{x})\|+\sup_{t\in [\tau,\varsigma)}  \|p_t(\tau,\eta,x)-p_t(\tau,\eta,\tilde{x})\|  <\varepsilon,
\end{align*}
which gives the continuity of $p_t(\tau,\eta,x)$ in $x$.
Moreover, by construction, we see that
$$
\mathcal{W}_s(\tau,x)=\{x+p_\tau(\tau,\eta,x):\eta\in \mathcal{X}_{\tau}^s\}
=\{\zeta+h_s(\tau,\zeta,x):\zeta\in \mathcal{X}_{\tau}^s \}
$$
is a $C^0$ leaf of the stable foliation.
For any $\hat{x}\in \mathcal{W}_s(\tau,x)$, one can easily verify that $x(t,\tau,\hat{x})\in \mathcal{W}_s(t,x(t,\tau,x))$, and thus,
$x(t,\tau,\mathcal{W}_s(\tau,x))\subset \mathcal{W}_s(t,x(t,\tau,x))$.
This gives the invariance of the stable foliation,  and we complete   the proof of Lemma \ref{thm-foli}.
\end{proof}

\section{Proof of  Proposition  \ref{V1-thm}--\ref{thm-ud}}\label{App-4}

\begin{proof}[Proof of Proposition \ref{V1-thm}.]
In what follows, we use the notation $x=x_s+x_u$ where $x_s\in \mathcal{X}_t^s$ and $x_u\in \mathcal{X}_t^u$.
For each $(t,x)\in \mathbb{R}\times \mathbb{R}^n$, we put
$$
V(t,x)=-V_s(t,x_s)+V_u(t,x_u),
$$
where
$$
V_s(t,x_s)=\sup_{\kappa \ge t}\left\{\|\Psi(\kappa,t)x_s\|\left(\frac{\mu(\kappa)}{\mu(t)}\right)^{-\lambda_s} \right\} \quad  \mathrm{and} \quad
V_u(t,x_u)=\sup_{\kappa \le t}\left\{\|\Psi(\kappa,t)x_u\|\left(\frac{\mu(\kappa)}{\mu(t)}\right)^{-\lambda_u} \right\}.
$$
Then by \eqref{u-dich} we see that
$$
V_s(t,x_s) \le D \mu(t)^{\mathrm{sign}(t)\nu}\|x_s\| \quad  \mathrm{and} \quad V_u(t,x_u) \le  D \mu(t)^{\mathrm{sign}(t)\omega}\|x_u\|,
$$
and thus,
\begin{align*}
|V(t,x)|  \le&~ D \mu(t)^{\mathrm{sign}(t)\nu}\|x_s\| + D \mu(t)^{\mathrm{sign}(t)\omega}\|x_u\|
\\
\le&~ 2D \max\left\{\mu(t)^{\mathrm{sign}(t)\nu},  \mu(t)^{\mathrm{sign}(t)\omega}\right\} \|x\|,
\end{align*}
which implies that \eqref{u-grow} holds.
Moreover, we see that for $t\ge \tau$,
$$
V_s(t,\Psi(t,\tau)x_s) \le \left(\frac{\mu(t)}{\mu(\tau)}\right)^{\lambda_s} V_s(\tau,x_s) \le V_s(\tau,x_s)
$$
and
$$
V_u(t,\Psi(t,\tau)x_u) \ge \left(\frac{\mu(t)}{\mu(\tau)}\right)^{\lambda_u} V_u(\tau,x_u) \ge V_u(\tau,x_u).
$$
Hence,
\begin{align*}
V(t,\Psi(t,\tau)x) =&  -V_s(t,\Psi(t,\tau)x_s) + V_u(t,\Psi(t,\tau)x_u)
\\
\ge&  -V_s(\tau,x_s)+V_u(\tau,x_u)=V(\tau,x),
\end{align*}
which indicates that condition (2) of Definition \ref{Lp-func} holds.

If $x\in \mathcal{E}_\tau^u$, then for $t\ge \tau$ we have
\begin{align*}
V(t,\Psi(t,\tau)x) =& -V_s(t,\Psi(t,\tau)x_s) + V_u(t,\Psi(t,\tau)x_u)
\\
=& -\sup_{\kappa\ge t} \left\{\|\Psi(\kappa,t)\Psi(t,\tau)x_s\|\left(\frac{\mu(\kappa)}{\mu(t)}\right)^{-\lambda_s} \right\}
\\
&+\sup _{\kappa\le t} \left\{\|\Psi(\kappa,t)\Psi(t,\tau)x_u\|\left(\frac{\mu(\kappa)}{\mu(t)}\right)^{-\lambda_u} \right\}
\\
=& - \left(\frac{\mu(t)}{\mu(\tau)}\right)^{\lambda_s}\sup_{\kappa\ge \tau} \left\{\|\Psi(\kappa,\tau)x_s\|\left(\frac{\mu(\kappa)}{\mu(\tau)}\right)^{-\lambda_s} \right\}
\\
&+ \left(\frac{\mu(t)}{\mu(\tau)}\right)^{\lambda_u}\sup_{\kappa\ge \tau} \left\{\|\Psi(\kappa,\tau)x_u\|\left(\frac{\mu(\kappa)}{\mu(\tau)}\right)^{-\lambda_u} \right\}
\\
\ge& \left(\frac{\mu(t)}{\mu(\tau)}\right)^{\lambda_u}
    \left(-\sup_{\kappa\ge \tau} \left\{\|\Psi(\kappa,\tau)x_s\|\left(\frac{\mu(\kappa)}{\mu(\tau)}\right)^{-\lambda_s} \right\}\right.
\\
&+ \left. \sup_{\kappa\ge \tau} \left\{\|\Psi(\kappa,\tau)x_u\|\left(\frac{\mu(\kappa)}{\mu(\tau)}\right)^{-\lambda_u} \right\} \right)
\\
=& \left(\frac{\mu(t)}{\mu(\tau)}\right)^{\lambda_u} V(\tau,x),
\end{align*}
and condition (1) of Definition \ref{str-V} holds with $\alpha=-\lambda_u$.
Similarly, if $x\in \mathcal{E}_\tau^s$, then for $t\ge\tau$ we have
\begin{align*}
|V(t,\Psi(t,\tau)x)| =&~  V_s(t,\Psi(t,\tau)x_s)- V_u(t,\Psi(t,\tau)x_u)
\\
=&  \left(\frac{\mu(t)}{\mu(\tau)}\right)^{\lambda_s} \sup_{\kappa\ge \tau}\left\{\|\Psi(\kappa,\tau)x_s\|\left(\frac{\mu(\kappa)}{\mu(\tau)}\right)^{-\lambda_s} \right\}
\\
& -  \left(\frac{\mu(t)}{\mu(\tau)}\right)^{\lambda_u} \sup_{\kappa\le \tau}\left\{\|\Psi(\kappa,\tau)x_u\|\left(\frac{\mu(\kappa)}{\mu(\tau)}\right)^{-\lambda_u} \right\}
\\
\le &   \left(\frac{\mu(t)}{\mu(\tau)}\right)^{\lambda_s}
       \left(\sup_{\kappa\ge \tau}\left\{\|\Psi(\kappa,\tau)x_s\|\left(\frac{\mu(\kappa)}{\mu(\tau)}\right)^{-\lambda_s} \right\} \right.
\\
& - \left. \sup_{\kappa\le \tau}\left\{\|\Psi(\kappa,\tau)x_u\|\left(\frac{\mu(\kappa)}{\mu(\tau)}\right)^{-\lambda_u} \right\}  \right)
\\
=&   \left(\frac{\mu(t)}{\mu(\tau)}\right)^{\lambda_s} |V(\tau,x)|,
\end{align*}
and condition (2) of Definition \ref{str-V} holds with $\beta=\lambda_s$.

It remains to show that condition (3) of Definition \ref{str-V} is valid. In fact, given a $\tau\in\mathbb{R}$,  if $x\in\mathcal{E}_\tau^u$  then
\begin{align*}
V(\tau,x) \ge&~ V(\tau,x)- V(\tau-1,\Psi(\tau-1,\tau)x)
\\
=& -V_s(\tau,x_s)+V_s(\tau-1,\Psi(\tau-1,\tau)x_s)
\\
&+ V_u(\tau,x_u)- V_u(\tau-1,\Psi(\tau-1,\tau)x_u),
\end{align*}
where
\begin{align*}
 & -V_s(\tau,x_s)+V_s(\tau-1,\Psi(\tau-1,\tau)x_s)
\\
&~ = -\sup_{\kappa\ge \tau} \left\{\|\Psi(\kappa,\tau)x_s\| \left(\frac{\mu(\kappa)}{\mu(\tau)}\right)^{-\lambda_s} \right\}
  +\left(\frac{\mu(\tau-1)}{\mu(\tau)}\right)^{\lambda_s}\sup_{\kappa\ge \tau-1} \left\{\|\Psi(\kappa,\tau)x_s\| \left(\frac{\mu(\kappa)}{\mu(\tau)}\right)^{-\lambda_s} \right\}
\\
&~ \ge \left(\left(\frac{\mu(\tau-1)}{\mu(\tau)}\right)^{\lambda_s}-1\right) \sup_{\kappa\ge \tau}
   \left\{\|\Psi(\kappa,\tau)x_s\| \left(\frac{\mu(\kappa)}{\mu(\tau)}\right)^{-\lambda_s} \right\} \ge \left(\left(\frac{\mu(\tau-1)}{\mu(\tau)}\right)^{\lambda_s}-1\right) \|x_s\|
\end{align*}
and
\begin{align*}
 & V_u(\tau,x_u)- V_u(\tau-1,\Psi(\tau-1,\tau)x_u)
\\
&~ = \sup_{\kappa\le \tau} \left\{\|\Psi(\kappa,\tau)x_u\| \left(\frac{\mu(\kappa)}{\mu(\tau)}\right)^{-\lambda_u} \right\}
 -\left(\frac{\mu(\tau-1)}{\mu(\tau)}\right)^{\lambda_u}\sup_{\kappa\le \tau-1} \left\{\|\Psi(\kappa,\tau)x_u\| \left(\frac{\mu(\kappa)}{\mu(\tau)}\right)^{-\lambda_u} \right\}
\\
&~ \ge \left( 1- \left(\frac{\mu(\tau-1)}{\mu(\tau)}\right)^{\lambda_u}\right) \sup_{\kappa\le \tau} \left\{\|\Psi(\kappa,\tau)x_u\| \left(\frac{\mu(\kappa)}{\mu(\tau)}\right)^{-\lambda_u} \right\}
\ge  \left( 1- \left(\frac{\mu(\tau-1)}{\mu(\tau)}\right)^{\lambda_u}\right) \|x_u\|.
\end{align*}
For any given $\tau\in\mathbb{R}$, letting
$$
C_u:=C_u(\tau) = \min\left\{\left(\frac{\mu(\tau-1)}{\mu(\tau)}\right)^{\lambda_s}-1,  1- \left(\frac{\mu(\tau-1)}{\mu(\tau)}\right)^{\lambda_u}\right\},
$$
we have
$$
V(\tau,x) \ge C_u(\|x_s\|+\|x_u\|) \ge C_u \|x\| \ge C_u  \mu(\tau)^{-\mathrm{sign}(\tau)\epsilon}\|x\|.
$$
On the other hand, if $x\in\mathcal{E}_\tau^s$ then
\begin{align*}
|V(\tau,x)| \ge&  |V(\tau,x)|- |V(\tau+1,\Psi(\tau+1,\tau)x)|
\\
=&  V_s(\tau,x_s)-V_s(\tau+1,\Psi(\tau+1,\tau)x_s)
\\
&- V_u(\tau,x_u)+ V_u(\tau+1,\Psi(\tau+1,\tau)x_u),
\end{align*}
where
\begin{align*}
 &V_s(\tau,x_s)-V_s(\tau+1,\Psi(\tau+1,\tau)x_s)
\\
&=~ \sup_{\kappa\ge \tau} \left\{\|\Psi(\kappa,\tau)x_s\| \left(\frac{\mu(\kappa)}{\mu(\tau)}\right)^{-\lambda_s} \right\}
  -\left(\frac{\mu(\tau+1)}{\mu(\tau)}\right)^{\lambda_s}\sup_{\kappa\ge \tau+1} \left\{\|\Psi(\kappa,\tau)x_s\| \left(\frac{\mu(\kappa)}{\mu(\tau)}\right)^{-\lambda_s} \right\}
\\
&\ge ~ \left(1-\left(\frac{\mu(\tau+1)}{\mu(\tau)}\right)^{\lambda_s}\right)  \sup_{\kappa\ge \tau} \left\{\|\Psi(\kappa,\tau)x_s\| \left(\frac{\mu(\kappa)}{\mu(\tau)}\right)^{-\lambda_s} \right\} \ge \left(1-\left(\frac{\mu(\tau+1)}{\mu(\tau)}\right)^{\lambda_s}\right) \|x_s\|
\end{align*}
and
\begin{align*}
 &- V_u(\tau,x_u)+ V_u(\tau+1,\Psi(\tau+1,\tau)x_u)
\\
&=~ - \sup_{\kappa\le \tau} \left\{\|\Psi(\kappa,\tau)x_u\| \left(\frac{\mu(\kappa)}{\mu(\tau)}\right)^{-\lambda_u} \right\}
 + \left(\frac{\mu(\tau+1)}{\mu(\tau)}\right)^{\lambda_u}\sup_{\kappa\le \tau+1} \left\{\|\Psi(\kappa,\tau)x_u\| \left(\frac{\mu(\kappa)}{\mu(\tau)}\right)^{-\lambda_u} \right\}
\\
&\ge ~ \left(\left(\frac{\mu(\tau+1)}{\mu(\tau)}\right)^{\lambda_u}-1\right)  \sup_{\kappa\le \tau} \left\{\|\Psi(\kappa,\tau)x_u\| \left(\frac{\mu(\kappa)}{\mu(\tau)}\right)^{-\lambda_u} \right\} \ge \left(\left(\frac{\mu(\tau+1)}{\mu(\tau)}\right)^{\lambda_u}-1\right) \|x_u\|.
\end{align*}
For any given $\tau\in\mathbb{R}$, letting
$$
C_s:=C_s(\tau)=\min\left\{1-\left(\frac{\mu(\tau+1)}{\mu(\tau)}\right)^{\lambda_s},\left(\frac{\mu(\tau+1)}{\mu(\tau)}\right)^{\lambda_u}-1 \right\},
$$
we have
$$
V(\tau,x) \ge C_s (\|x_s\|+\|x_u\|) \ge C_s \|x\|\ge C_s \mu(\tau)^{-\mathrm{sign}(\tau)\epsilon}\|x\|.
$$
Consequently, condition (3)  of Definition \ref{str-V} is proved provided that $C\ge \max\{ 1/C_s,1/C_u\}$. This completes the proof.
\end{proof}

\begin{proof}[Proof of Proposition \ref{V-thm2}.]
For each $(t,x)\in\mathbb{R}\times \mathbb{R}^n$, consider the function
\begin{align}\label{QV-def}
\text{$V(t,x)= -\mathrm{sign}\, U(t,x) \sqrt{|U(t,x)|}$,  where $U(t,x)=\langle S(t)x,x\rangle$}
\end{align}
and
\begin{align}\label{S-linOpe}
S(t)=&~ \int_t^{+\infty} \big(\Psi(\kappa,t)\pi(t)\big)^{*} \Psi(\kappa,t)\pi(t)
                \left(\frac{\mu(\kappa)}{\mu(t)}\right)^{-2(\lambda_s+\eta)} \frac{\mu'(\kappa)}{\mu(\kappa)} \, d\kappa
\notag \\
&~ - \int_{-\infty}^t  \big(\Psi(\kappa,t)(id-\pi(t))\big)^{*} \Psi(\kappa,t)(id-\pi(t))
             \left(\frac{\mu(t)}{\mu(\kappa)}\right)^{2(\lambda_u-\eta)} \frac{\mu'(\kappa)}{\mu(\kappa)} \, d\kappa,
\end{align}
for some constant $\eta>0$ such that $\eta<\min\{-\lambda_s,\lambda_u\}$.
Obviously, $S(t)$ is symmetric for each $t\in\mathbb{R}$.
Furthermore, we have
$$
\text{$U(t,x)<0$ for   $x\in \mathcal{X}_t^u\backslash \{0\}$ and $U(t,x)>0$ for   $x\in \mathcal{X}_t^s\backslash \{0\}$.  }
$$
Thus, $S(t)$ is invertible due to $\mathcal{X}_t^u \oplus \mathcal{X}_t^s = \mathbb{R}^n$ for any $t\in\mathbb{R}$.

In order to obtain condition (2) of Definition \ref{Lp-func}, we see that for every $t\ge \tau$,
\begin{align*}
 U(t,\Psi(t,\tau)x)=& \int_t^{+\infty} \|\Psi(\kappa,\tau)\pi(\tau)x\|^2  \left(\frac{\mu(\kappa)}{\mu(t)}\right)^{-2(\lambda_s+\eta)} \frac{\mu'(\kappa)}{\mu(\kappa)} \, d\kappa
\\
&- \int_{-\infty}^t \|\Psi(\kappa,\tau)(id-\pi(\tau))x\|^2  \left(\frac{\mu(t)}{\mu(\kappa)}\right)^{2(\lambda_u-\eta)} \frac{\mu'(\kappa)}{\mu(\kappa)} \, d\kappa
\\
\le& \int_{\tau}^{+\infty} \|\Psi(\kappa,\tau)\pi(\tau)x\|^2  \left(\frac{\mu(\kappa)}{\mu(\tau)}\right)^{-2(\lambda_s+\eta)} \frac{\mu'(\kappa)}{\mu(\kappa)} \, d\kappa
\\
&- \int_{-\infty}^{\tau} \|\Psi(\kappa,\tau)(id-\pi(\tau))x\|^2  \left(\frac{\mu(\tau)}{\mu(\kappa)}\right)^{2(\lambda_u-\eta)} \frac{\mu'(\kappa)}{\mu(\kappa)} \, d\kappa
\\
=&~ U(\tau,x),
\end{align*}
since
\begin{align*}
& \int_t^{+\infty} \|\Psi(\kappa,\tau)\pi(\tau)x\|^2  \left(\frac{\mu(\kappa)}{\mu(t)}\right)^{-2(\lambda_s+\eta)} \frac{\mu'(\kappa)}{\mu(\kappa)} \, d\kappa
\\
&\le \left(\frac{\mu(t)}{\mu(\tau)}\right)^{2(\lambda_s+\eta)} \int_{\tau}^{+\infty}
  \|\Psi(\kappa,\tau)\pi(\tau)x\|^2  \left(\frac{\mu(\kappa)}{\mu(\tau)}\right)^{-2(\lambda_s+\eta)} \frac{\mu'(\kappa)}{\mu(\kappa)} \, d\kappa
\\
&\le  \int_{\tau}^{+\infty}
  \|\Psi(\kappa,\tau)\pi(\tau)x\|^2  \left(\frac{\mu(\kappa)}{\mu(\tau)}\right)^{-2(\lambda_s+\eta)} \frac{\mu'(\kappa)}{\mu(\kappa)} \, d\kappa, \quad t\ge \tau,
\end{align*}
and
\begin{align*}
& \int_{-\infty}^t \|\Psi(\kappa,\tau)(id-\pi(\tau))x\|^2  \left(\frac{\mu(t)}{\mu(\kappa)}\right)^{2(\lambda_u-\eta)} \frac{\mu'(\kappa)}{\mu(\kappa)} \, d\kappa
\\
& \ge  \left(\frac{\mu(t)}{\mu(\tau)}\right)^{2(\lambda_u-\eta)}  \int_{-\infty}^\tau
\|\Psi(\kappa,\tau)(id-\pi(\tau))x\|^2  \left(\frac{\mu(\tau)}{\mu(\kappa)}\right)^{2(\lambda_u-\eta)} \frac{\mu'(\kappa)}{\mu(\kappa)} \, d\kappa
\\
&\ge  \int_{-\infty}^\tau
\|\Psi(\kappa,\tau)(id-\pi(\tau))x\|^2  \left(\frac{\mu(\tau)}{\mu(\kappa)}\right)^{2(\lambda_u-\eta)} \frac{\mu'(\kappa)}{\mu(\kappa)} \, d\kappa, \quad t\ge \tau.
\end{align*}
Therefore by \eqref{QV-def}, the required result is proved.

Next, we show that $V(t,x)$ given by \eqref{QV-def} is a strict Lyapunov function.
It follows from \eqref{u-dich} that
\begin{align}\label{UUU-1}
|U(t,x)| \le& \int_t^{+\infty} \|\Psi(\kappa,t)\pi(t)x\|^2  \left(\frac{\mu(\kappa)}{\mu(t)}\right)^{-2(\lambda_s+\eta)} \frac{\mu'(\kappa)}{\mu(\kappa)} \, d\kappa
\notag \\
&+ \int_{-\infty}^t \|\Psi(\kappa,t)(id-\pi(t))x\|^2  \left(\frac{\mu(t)}{\mu(\kappa)}\right)^{2(\lambda_u-\eta)} \frac{\mu'(\kappa)}{\mu(\kappa)} \, d\kappa
\notag \\
\le&~ D^2 \mu(t)^{\mathrm{sign}(t)\nu}\|x\|^2 \int_t^{+\infty} \left(\frac{\mu(\kappa)}{\mu(t)}\right)^{-2\eta} \frac{\mu'(\kappa)}{\mu(\kappa)} \, d\kappa
\notag \\
&+ D^2 \mu(t)^{\mathrm{sign}(t)\omega}\|x\|^2  \int_{-\infty}^t  \left(\frac{\mu(t)}{\mu(\kappa)}\right)^{-2\eta} \frac{\mu'(\kappa)}{\mu(\kappa)} \, d\kappa
\notag \\
\le&~ \frac{D^2}{\eta} \max\left\{\mu(t)^{\mathrm{sign}(t)\nu},  \mu(t)^{\mathrm{sign}(t)\omega}\right\}\|x\|^2,
\end{align}
which implies that \eqref{u-grow} holds.

Now assume that $x\in\mathcal{E}_\tau^s$. Then for each $t\ge \tau$, we see that
\begin{align*}
U(t,\Psi(t,\tau)x)=& \int_t^{+\infty} \|\Psi(\kappa,\tau)x_s\|^2  \left(\frac{\mu(\kappa)}{\mu(t)}\right)^{-2(\lambda_s+\eta)} \frac{\mu'(\kappa)}{\mu(\kappa)} \, d\kappa
\\
&- \int_{-\infty}^t \|\Psi(\kappa,\tau)x_u\|^2  \left(\frac{\mu(t)}{\mu(\kappa)}\right)^{2(\lambda_u-\eta)} \frac{\mu'(\kappa)}{\mu(\kappa)} \, d\kappa
\\
\le&   \left(\frac{\mu(t)}{\mu(\tau)}\right)^{2(\lambda_s+\eta)} \int_{\tau}^{+\infty}
  \|\Psi(\kappa,\tau)x_s\|^2  \left(\frac{\mu(\kappa)}{\mu(\tau)}\right)^{-2(\lambda_s+\eta)} \frac{\mu'(\kappa)}{\mu(\kappa)} \, d\kappa
\\
&- \left(\frac{\mu(t)}{\mu(\tau)}\right)^{2(\lambda_u-\eta)}  \int_{-\infty}^\tau
\|\Psi(\kappa,\tau)x_u\|^2  \left(\frac{\mu(\tau)}{\mu(\kappa)}\right)^{2(\lambda_u-\eta)} \frac{\mu'(\kappa)}{\mu(\kappa)} \, d\kappa
\\
\le&  \left(\frac{\mu(t)}{\mu(\tau)}\right)^{2(\lambda_s+\eta)} \left(\int_{\tau}^{+\infty}
  \|\Psi(\kappa,\tau)x_s\|^2  \left(\frac{\mu(\kappa)}{\mu(\tau)}\right)^{-2(\lambda_s+\eta)} \frac{\mu'(\kappa)}{\mu(\kappa)} \, d\kappa \right.
\\
&- \left.  \int_{-\infty}^\tau
\|\Psi(\kappa,\tau)x_u\|^2  \left(\frac{\mu(\tau)}{\mu(\kappa)}\right)^{2(\lambda_u-\eta)} \frac{\mu'(\kappa)}{\mu(\kappa)} \, d\kappa \right)
\\
=& \left(\frac{\mu(t)}{\mu(\tau)}\right)^{2(\lambda_s+\eta)} U(\tau,x),
\end{align*}
because of $\lambda_u-\eta>\lambda_s+\eta$, and thus condition (2) of  Definition \ref{str-V} holds with $\eta\ge \lambda_s+\eta$.
 Similarly in the case of $x\in\mathcal{E}_\tau^u$, for each $t\ge \tau$ we see that
\begin{align*}
|U(t,\Psi(t,\tau)x)| =& -\int_t^{+\infty} \|\Psi(\kappa,\tau)x_s\|^2  \left(\frac{\mu(\kappa)}{\mu(t)}\right)^{-2(\lambda_s+\eta)} \frac{\mu'(\kappa)}{\mu(\kappa)} \, d\kappa
\\
&+ \int_{-\infty}^t \|\Psi(\kappa,\tau)x_u\|^2  \left(\frac{\mu(t)}{\mu(\kappa)}\right)^{2(\lambda_u-\eta)} \frac{\mu'(\kappa)}{\mu(\kappa)} \, d\kappa
\\
\ge& \left(\frac{\mu(t)}{\mu(\tau)}\right)^{2(\lambda_u-\eta)} \left(- \int_{\tau}^{+\infty}
  \|\Psi(\kappa,\tau)x_s\|^2  \left(\frac{\mu(\kappa)}{\mu(\tau)}\right)^{-2(\lambda_s+\eta)} \frac{\mu'(\kappa)}{\mu(\kappa)} \, d\kappa  \right.
\\
&+ \left. \int_{-\infty}^\tau
\|\Psi(\kappa,\tau)x_u\|^2  \left(\frac{\mu(\tau)}{\mu(\kappa)}\right)^{2(\lambda_u-\eta)} \frac{\mu'(\kappa)}{\mu(\kappa)} \, d\kappa \right)
\\
=& \left(\frac{\mu(t)}{\mu(\tau)}\right)^{2(\lambda_u-\eta)} |U(\tau,x)|,
\end{align*}
and thus condition (1) of  Definition \ref{str-V} holds with $\alpha \ge -\lambda_u+\eta$.

Finally, we claim that condition (3) of  Definition \ref{str-V} is valid.
If $x\in\mathcal{E}_\tau^s$ then
\begin{align}\label{VVV-1}
U(\tau,x) \ge &~ U(\tau,x) - U(\tau-1,\Psi(\tau-1,\tau)x)
\notag \\
=& -U_s(\tau,x_s)+U_s(\tau-1,\Psi(\tau-1,\tau)x_s)
\notag \\
&+ U_u(\tau,x_u) - U_u(\tau-1,\Psi(\tau-1,\tau)x_u),
\end{align}
where
\begin{align*}
&-U_s(\tau,x_s)+U_s(\tau-1,\Psi(\tau-1,\tau)x_s)
\\
&= - \int_{\tau}^{+\infty}
  \|\Psi(\kappa,\tau)x_s\|^2  \left(\frac{\mu(\kappa)}{\mu(\tau)}\right)^{-2(\lambda_s+\eta)} \frac{\mu'(\kappa)}{\mu(\kappa)} \, d\kappa
\\
&\quad + \int_{\tau-1}^{+\infty}
  \|\Psi(\kappa,\tau)x_s\|^2  \left(\frac{\mu(\kappa)}{\mu(\tau-1)}\right)^{-2(\lambda_s+\eta)} \frac{\mu'(\kappa)}{\mu(\kappa)} \, d\kappa
\end{align*}
and
\begin{align*}
& U_u(\tau,x_u) - U_u(\tau-1,\Psi(\tau-1,\tau)x_u)
\\
&= \int_{-\infty}^\tau \|\Psi(\kappa,\tau)x_u\|^2  \left(\frac{\mu(\tau)}{\mu(\kappa)}\right)^{2(\lambda_u-\eta)} \frac{\mu'(\kappa)}{\mu(\kappa)} \, d\kappa
\\
&\quad - \int_{-\infty}^{\tau-1} \|\Psi(\kappa,\tau)x_u\|^2  \left(\frac{\mu(\tau-1)}{\mu(\kappa)}\right)^{2(\lambda_u-\eta)} \frac{\mu'(\kappa)}{\mu(\kappa)} \, d\kappa.
\end{align*}
In view of \eqref{bou-u}, we obtain that
\begin{align}\label{VVV-2}
& -U_s(\tau,x_s)+U_s(\tau-1,\Psi(\tau-1,\tau)x_s)
\notag \\
&=  - \int_{\tau}^{+\infty}
  \|\Psi(\kappa,\tau)x_s\|^2  \left(\frac{\mu(\kappa)}{\mu(\tau)}\right)^{-2(\lambda_s+\eta)} \frac{\mu'(\kappa)}{\mu(\kappa)} \, d\kappa
\notag \\
&\quad +  \left(\frac{\mu(\tau)}{\mu(\tau-1)}\right)^{-2(\lambda_s+\eta)}    \int_{\tau-1}^{+\infty}
  \|\Psi(\kappa,\tau)x_s\|^2  \left(\frac{\mu(\kappa)}{\mu(\tau)}\right)^{-2(\lambda_s+\eta)} \frac{\mu'(\kappa)}{\mu(\kappa)} \, d\kappa
\notag \\
&\ge \left(\left(\frac{\mu(\tau)}{\mu(\tau-1)}\right)^{-2(\lambda_s+\eta)} -1 \right)   \int_{\tau}^{+\infty}
  \|\Psi(\kappa,\tau)x_s\|^2  \left(\frac{\mu(\kappa)}{\mu(\tau)}\right)^{-2(\lambda_s+\eta)} \frac{\mu'(\kappa)}{\mu(\kappa)} \, d\kappa
\notag \\
&\ge \left(\left(\frac{\mu(\tau)}{\mu(\tau-1)}\right)^{-2(\lambda_s+\eta)} -1 \right)   \int_{\tau}^{\tau+c}
  \|\Psi(\kappa,\tau)x_s\|^2  \left(\frac{\mu(\kappa)}{\mu(\tau)}\right)^{-2(\lambda_s+\eta)} \frac{\mu'(\kappa)}{\mu(\kappa)} \, d\kappa
\notag \\
&\ge \frac{1}{\tilde{D}^2}\mu(|\tau|)^{-2\tilde{\lambda}}\left(\left(\frac{\mu(\tau)}{\mu(\tau-1)}\right)^{-2(\lambda_s+\eta)} -1 \right)  \|x_s\|^2
   \int_{\tau}^{\tau+c}\left(\frac{\mu(\kappa)}{\mu(\tau)}\right)^{-2(\lambda_s+\eta)} \frac{\mu'(\kappa)}{\mu(\kappa)} \, d\kappa
\notag \\
&\ge \frac{1}{\tilde{D}^2}\mu(|\tau|)^{-2\tilde{\lambda}}\left(\left(\frac{\mu(\tau)}{\mu(\tau-1)}\right)^{-2(\lambda_s+\eta)} -1 \right)
  \log\frac{\mu(\tau+c)}{\mu(\tau)} \|x_s\|^2,
\end{align}
here the   last inequality holds  due to
$$
\int_{\tau}^{\tau+c}\left(\frac{\mu(\kappa)}{\mu(\tau)}\right)^{-2(\lambda_s+\eta)} \frac{\mu'(\kappa)}{\mu(\kappa)} \, d\kappa
\ge \int_{\tau}^{\tau+c}  \frac{\mu'(\kappa)}{\mu(\kappa)} \, d\kappa =  \log\frac{\mu(\tau+c)}{\mu(\tau)}>0.
$$
Similar to \eqref{VVV-2}, we obtain from \eqref{bou-u} that
\begin{align}\label{VVV-3}
& U_u(\tau,x_u) - U_u(\tau-1,\Psi(\tau-1,\tau)x_u)
\notag \\
&\ge \left(1-\left(\frac{\mu(\tau-1)}{\mu(\tau)}\right)^{2(\lambda_u-\eta)} \right)
\int_{-\infty}^{\tau-1} \|\Psi(\kappa,\tau)x_u\|^2  \left(\frac{\mu(\tau)}{\mu(\kappa)}\right)^{2(\lambda_u-\eta)} \frac{\mu'(\kappa)}{\mu(\kappa)} \, d\kappa
\notag \\
&\ge  \left(1-\left(\frac{\mu(\tau-1)}{\mu(\tau)}\right)^{2(\lambda_u-\eta)} \right)
\int_{\tau-1-c}^{\tau-1} \|\Psi(\kappa,\tau)x_u\|^2  \left(\frac{\mu(\tau)}{\mu(\kappa)}\right)^{2(\lambda_u-\eta)} \frac{\mu'(\kappa)}{\mu(\kappa)} \, d\kappa
\notag \\
&\ge \frac{1}{\tilde{D}^2}\mu(|\tau|)^{-2\tilde{\lambda}} \left(1-\left(\frac{\mu(\tau-1)}{\mu(\tau)}\right)^{2(\lambda_u-\eta)} \right)
 \log\frac{\mu(\tau-1)}{\mu(\tau-1-c)} \|x_u\|^2,
\end{align}
where
$$
\int_{\tau-1-c}^{\tau-1}   \left(\frac{\mu(\tau)}{\mu(\kappa)}\right)^{2(\lambda_u-\eta)} \frac{\mu'(\kappa)}{\mu(\kappa)} \, d\kappa
\ge \int_{\tau-1-c}^{\tau-1} \frac{\mu'(\kappa)}{\mu(\kappa)} \, d\kappa= \log\frac{\mu(\tau-1)}{\mu(\tau-1-c)}>0.
$$
For any given $\tau\in\mathbb{R}$, let
\begin{align*}
C_s:=C_s(\tau)=\frac{1}{\tilde{D}^2}&  \min\left\{\left(\left(\frac{\mu(\tau)}{\mu(\tau-1)}\right)^{-2(\lambda_s+\eta)} -1 \right)
  \log\frac{\mu(\tau+c)}{\mu(\tau)}, \right.
\\
&\qquad\quad  \left.  \left(1-\left(\frac{\mu(\tau-1)}{\mu(\tau)}\right)^{2(\lambda_u-\eta)} \right)
 \log\frac{\mu(\tau-1)}{\mu(\tau-1-c)}\right\}.
\end{align*}
Then by \eqref{VVV-1}, \eqref{VVV-2} and \eqref{VVV-3}, we conclude that
\begin{align}\label{VVV-4}
|V(\tau,x)|^2=&~ U(\tau,x) \ge C_s \mu(|\tau|)^{-2\tilde{\lambda}} \left(\|x_s\|^2+\|x_u\|^2\right)
\notag \\
\ge&~ C_s \mu(|\tau|)^{-2\tilde{\lambda}} \left(\frac{\|x_s\|+\|x_u\|}{2}\right)^2 \ge \frac{C_s}{4}\mu(|\tau|)^{-2\tilde{\lambda}}\|x\|^2.
\end{align}
On the other hand, if $x\in\mathcal{E}_\tau^u$ then
\begin{align}\label{VVV-5}
|U(\tau,x)| \ge&~ |U(\tau,x)|-|U(\tau+1,\Psi(\tau+1,\tau)x)|
\notag \\
=&~ U_s(\tau,x_s)- U_s(\tau+1,\Psi(\tau+1,\tau)x_s)
\notag \\
&~ - U_u(\tau,x_u) - U_u(\tau+1,\Psi(\tau+1,\tau)x_u),
\end{align}
where
\begin{align*}
&  U_s(\tau,x_s)- U_s(\tau+1,\Psi(\tau+1,\tau)x_s)
\\
&=   \int_{\tau}^{+\infty}
  \|\Psi(\kappa,\tau)x_s\|^2  \left(\frac{\mu(\kappa)}{\mu(\tau)}\right)^{-2(\lambda_s+\eta)} \frac{\mu'(\kappa)}{\mu(\kappa)} \, d\kappa
\\
&\quad - \int_{\tau+1}^{+\infty}
  \|\Psi(\kappa,\tau)x_s\|^2  \left(\frac{\mu(\kappa)}{\mu(\tau+1)}\right)^{-2(\lambda_s+\eta)} \frac{\mu'(\kappa)}{\mu(\kappa)} \, d\kappa
\end{align*}
and
\begin{align*}
& - U_u(\tau,x_u) - U_u(\tau+1,\Psi(\tau+1,\tau)x_u)
\\
&= -\int_{-\infty}^\tau \|\Psi(\kappa,\tau)x_u\|^2  \left(\frac{\mu(\tau)}{\mu(\kappa)}\right)^{2(\lambda_u-\eta)} \frac{\mu'(\kappa)}{\mu(\kappa)} \, d\kappa
\\
&\quad + \int_{-\infty}^{\tau+1} \|\Psi(\kappa,\tau)x_u\|^2  \left(\frac{\mu(\tau+1)}{\mu(\kappa)}\right)^{2(\lambda_u-\eta)} \frac{\mu'(\kappa)}{\mu(\kappa)} \, d\kappa.
\end{align*}
In view of \eqref{bou-u}, we deduce that
\begin{align}\label{VVV-6}
& U_s(\tau,x_s)- U_s(\tau+1,\Psi(\tau+1,\tau)x_s)
\notag \\
&= \int_{\tau}^{+\infty}
  \|\Psi(\kappa,\tau)x_s\|^2  \left(\frac{\mu(\kappa)}{\mu(\tau)}\right)^{-2(\lambda_s+\eta)} \frac{\mu'(\kappa)}{\mu(\kappa)} \, d\kappa
\notag \\
&\quad -  \left(\frac{\mu(\tau+1)}{\mu(\tau)}\right)^{2(\lambda_s+\eta)}      \int_{\tau+1}^{+\infty}
  \|\Psi(\kappa,\tau)x_s\|^2  \left(\frac{\mu(\kappa)}{\mu(\tau)}\right)^{-2(\lambda_s+\eta)} \frac{\mu'(\kappa)}{\mu(\kappa)} \, d\kappa
\notag \\
&\ge \left(1- \left(\frac{\mu(\tau+1)}{\mu(\tau)}\right)^{2(\lambda_s+\eta)} \right) \int_{\tau+1}^{+\infty}
  \|\Psi(\kappa,\tau)x_s\|^2  \left(\frac{\mu(\kappa)}{\mu(\tau)}\right)^{-2(\lambda_s+\eta)} \frac{\mu'(\kappa)}{\mu(\kappa)} \, d\kappa
\notag \\
&\ge \left(1- \left(\frac{\mu(\tau+1)}{\mu(\tau)}\right)^{2(\lambda_s+\eta)} \right) \int_{\tau+1}^{\tau+1+c}
  \|\Psi(\kappa,\tau)x_s\|^2  \left(\frac{\mu(\kappa)}{\mu(\tau)}\right)^{-2(\lambda_s+\eta)} \frac{\mu'(\kappa)}{\mu(\kappa)} \, d\kappa
\notag \\
&\ge~ \frac{1}{\tilde{D}^2} \mu(|\tau|)^{-2\tilde{\lambda}}  \left(1- \left(\frac{\mu(\tau+1)}{\mu(\tau)}\right)^{2(\lambda_s+\eta)} \right) \|x_s\|^2
      \int_{\tau+1}^{\tau+1+c}  \left(\frac{\mu(\kappa)}{\mu(\tau)}\right)^{-2(\lambda_s+\eta)} \frac{\mu'(\kappa)}{\mu(\kappa)} \, d\kappa
\notag \\
&\ge~  \frac{1}{\tilde{D}^2} \mu(|\tau|)^{-2\tilde{\lambda}}  \left(1- \left(\frac{\mu(\tau+1)}{\mu(\tau)}\right)^{2(\lambda_s+\eta)} \right)
        \log\frac{\mu(\tau+1+c)}{\mu(\tau+1)} \|x_s\|^2
\end{align}
and
\begin{align}\label{VVV-7}
& - U_u(\tau,x_u) - U_u(\tau+1,\Psi(\tau+1,\tau)x_u)
\notag \\
&\ge \left(\left(\frac{\mu(\tau+1)}{\mu(\tau)}\right)^{2(\lambda_u-\eta)}-1 \right)
       \int_{-\infty}^\tau \|\Psi(\kappa,\tau)x_u\|^2  \left(\frac{\mu(\tau)}{\mu(\kappa)}\right)^{2(\lambda_u-\eta)} \frac{\mu'(\kappa)}{\mu(\kappa)} \, d\kappa
\notag \\
&\ge~  \frac{1}{\tilde{D}^2} \mu(|\tau|)^{-2\tilde{\lambda}}  \left(\left(\frac{\mu(\tau+1)}{\mu(\tau)}\right)^{2(\lambda_u-\eta)}-1 \right) \|x_u\|^2
         \int_{\tau-c}^\tau \left(\frac{\mu(\tau)}{\mu(\kappa)}\right)^{2(\lambda_u-\eta)} \frac{\mu'(\kappa)}{\mu(\kappa)} \, d\kappa
\notag \\
&\ge~  \frac{1}{\tilde{D}^2} \mu(|\tau|)^{-2\tilde{\lambda}}  \left(\left(\frac{\mu(\tau+1)}{\mu(\tau)}\right)^{2(\lambda_u-\eta)}-1 \right)
  \log\frac{\mu(\tau)}{\mu(\tau-c)} \|x_u\|^2.
\end{align}
From \eqref{VVV-5}, \eqref{VVV-6} and \eqref{VVV-7}, we conclude that
\begin{align}\label{VVV-8}
|V(\tau,x)|^2=   |U(\tau,x)| \ge \frac{C_u}{4} \mu(|\tau|)^{-2\tilde{\lambda}}\|x\|^2,
\end{align}
where
\begin{align*}
C_u:=C_u(\tau)= \frac{1}{\tilde{D}^2} &  \min\left\{ \left(1- \left(\frac{\mu(\tau+1)}{\mu(\tau)}\right)^{2(\lambda_s+\eta)} \right)
        \log\frac{\mu(\tau+1+c)}{\mu(\tau+1)}, \right.
\\
&\qquad\quad  \left.   \left(\left(\frac{\mu(\tau+1)}{\mu(\tau)}\right)^{2(\lambda_u-\eta)}-1 \right)
  \log\frac{\mu(\tau)}{\mu(\tau-c)}\right\}.
\end{align*}
By \eqref{VVV-4} and \eqref{VVV-8}, we assert that condition (3) of Definition \ref{str-V} holds since one can choose an appropriate $\epsilon$ so that
$\mu(\tau)^{-\mathrm{sign}(\tau)\epsilon}=\mu(|\tau|)^{-\tilde{\lambda}}$ for each $\tau\in\mathbb{R}$.
Consequently, $V(\tau,x)$ given in \eqref{QV-def} is a quadratic strict Lyapunov function. This completes the proof.
\end{proof}

\begin{proof}[Proof of Proposition \ref{Pro18}.]
Noting that $S(t)$ is symmetric, by \eqref{UUU-1} we have
$$
\|S(t)\|=\sup_{x\ne 0} \frac{|U(t,x)|}{\|x\|^2}\le  \frac{D^2}{\eta} \max\left\{\mu(t)^{\mathrm{sign}(t)\nu},  \mu(t)^{\mathrm{sign}(t)\omega}\right\}.
$$
Recall that
\begin{align*}
S(t)
=&  \int_t^{+\infty} \big(\pi(\kappa)\Psi(\kappa,t)\big)^{*} \pi(\kappa)\Psi(\kappa,t)
                \left(\frac{\mu(\kappa)}{\mu(t)}\right)^{-2(\lambda_s+\eta)} \frac{\mu'(\kappa)}{\mu(\kappa)} \, d\kappa
\\
&  - \int_{-\infty}^t  \big((id-\pi(\kappa))\Psi(\kappa,t)\big)^{*} (id-\pi(\kappa))\Psi(\kappa,t)
             \left(\frac{\mu(t)}{\mu(\kappa)}\right)^{2(\lambda_u-\eta)} \frac{\mu'(\kappa)}{\mu(\kappa)} \, d\kappa.
\end{align*}
Since
$$
\text{$\frac{\partial}{\partial t}  \Psi(\kappa,t)=- \Psi(\kappa,t)A(t)$ \quad  and \quad
    $\frac{\partial}{\partial t}  \Psi(\kappa,t)^*=- A(t)^*\Psi(\kappa,t)^*$,}
$$
we obtain that $S(t)$ is $C^1$ in $t$ with derivative
\begin{align*}
S'(t)=& -\frac{\mu'(t)}{\mu(t)}-  \int_t^{+\infty}A(t)^* \big(\pi(\kappa)\Psi(\kappa,t)\big)^{*} \pi(\kappa)\Psi(\kappa,t)
                \left(\frac{\mu(\kappa)}{\mu(t)}\right)^{-2(\lambda_s+\eta)} \frac{\mu'(\kappa)}{\mu(\kappa)} \, d\kappa
\\
&- \int_t^{+\infty} \big(\pi(\kappa)\Psi(\kappa,t)\big)^{*} \pi(\kappa)\Psi(\kappa,t) A(t)
                \left(\frac{\mu(\kappa)}{\mu(t)}\right)^{-2(\lambda_s+\eta)} \frac{\mu'(\kappa)}{\mu(\kappa)} \, d\kappa
\\
& + 2(\lambda_s+\eta) \frac{\mu'(t)}{\mu(t)} \int_t^{+\infty} \big(\pi(\kappa)\Psi(\kappa,t)\big)^{*} \pi(\kappa)\Psi(\kappa,t)
                \left(\frac{\mu(\kappa)}{\mu(t)}\right)^{-2(\lambda_s+\eta)} \frac{\mu'(\kappa)}{\mu(\kappa)} \, d\kappa
\\
&-\frac{\mu'(t)}{\mu(t)}+ \int_{-\infty}^t  A(t)^* \big((id-\pi(\kappa))\Psi(\kappa,t)\big)^{*} (id-\pi(\kappa))\Psi(\kappa,t)
             \left(\frac{\mu(t)}{\mu(\kappa)}\right)^{2(\lambda_u-\eta)} \frac{\mu'(\kappa)}{\mu(\kappa)} \, d\kappa
\\
&+ \int_{-\infty}^t  \big((id-\pi(\kappa))\Psi(\kappa,t)\big)^{*} (id-\pi(\kappa))\Psi(\kappa,t)A(t)
             \left(\frac{\mu(t)}{\mu(\kappa)}\right)^{2(\lambda_u-\eta)} \frac{\mu'(\kappa)}{\mu(\kappa)} \, d\kappa
\\
&- 2(\lambda_u-\eta) \frac{\mu'(t)}{\mu(t)}  \int_{-\infty}^t  \big((id-\pi(\kappa))\Psi(\kappa,t)\big)^{*} (id-\pi(\kappa))\Psi(\kappa,t)
             \left(\frac{\mu(t)}{\mu(\kappa)}\right)^{2(\lambda_u-\eta)} \frac{\mu'(\kappa)}{\mu(\kappa)} \, d\kappa
\\
=& -\frac{2\mu'(t)}{\mu(t)} -A(t)^*S(t)-S(t)A(t)
\\
&+ 2(\lambda_s+\eta) \frac{\mu'(t)}{\mu(t)} \int_t^{+\infty} \big(\pi(\kappa)\Psi(\kappa,t)\big)^{*} \pi(\kappa)\Psi(\kappa,t)
                \left(\frac{\mu(\kappa)}{\mu(t)}\right)^{-2(\lambda_s+\eta)} \frac{\mu'(\kappa)}{\mu(\kappa)} \, d\kappa
\\
&- 2(\lambda_u-\eta) \frac{\mu'(t)}{\mu(t)}  \int_{-\infty}^t  \big((id-\pi(\kappa))\Psi(\kappa,t)\big)^{*} (id-\pi(\kappa))\Psi(\kappa,t)
             \left(\frac{\mu(t)}{\mu(\kappa)}\right)^{2(\lambda_u-\eta)} \frac{\mu'(\kappa)}{\mu(\kappa)} \, d\kappa
\\
\le& -\frac{2\mu'(t)}{\mu(t)} -A(t)^*S(t)-S(t)A(t) + 2 (\lambda_u-\eta) \frac{\mu'(t)}{\mu(t)}S(t).
\end{align*}
Therefore, we have
$$
S'(t)+A(t)^*S(t)+S(t)A(t) \le  -\frac{2\mu'(t)}{\mu(t)} (id+ \lambda_u S(t)),
$$
which proves \eqref{S-b}.

From \eqref{VVV-4}, one knows that if $x\in\mathcal{E}_t^s$, then
$$
U(t,x)=\langle S(t)x,x \rangle \ge \frac{C_s}{4} \mu(|t|)^{-2\tilde{\lambda}} \langle x,x \rangle,
$$
which implies that $S(t)\ge  \frac{C_s}{4} \mu(|t|)^{-2\tilde{\lambda}}$ on $\mathcal{E}_t^s$ by Remark \ref{Rem-4}. Similarly, by \eqref{VVV-8}, if $x\in\mathcal{E}_t^u$ then
$$
U(t,x)=-\langle S(t)x,x \rangle \ge \frac{C_u}{4} \mu(|t|)^{-2\tilde{\lambda}} \langle x,x \rangle,
$$
and thus $-S(t) \ge  \frac{C_u}{4} \mu(|t|)^{-2\tilde{\lambda}}$ on $\mathcal{E}_t^u$. This completes the proof.
\end{proof}

\begin{proof}[Proof of Proposition \ref{thm-str}.]
It follows from Definition \ref{str-V}-(3)  that $\mathcal{G}_\tau^u\subset \mathcal{C}^u(V_\tau)$ and $\mathcal{G}_\tau^s\subset \mathcal{C}^s(V_\tau)$, thus the function
$V_\tau$ is positive in $\mathcal{G}_\tau^u\backslash \{0\}$ and negative in $\mathcal{G}_\tau^s\backslash \{0\}$.
For $x\in\mathcal{G}_\tau^s$ and $t\ge \tau$, we obtain from \eqref{u-grow} and Definition \ref{str-V}-(2)(3) that
\begin{align*}
\|\Psi(t,\tau)x\| \le &~ C \mu(t)^{\mathrm{sign}(t)\epsilon}|V(t,\Psi(t,\tau)x)| \le C  \left( \frac{\mu(t)}{\mu(\tau)}\right)^{\beta} \mu(t)^{\mathrm{sign}(t)\epsilon}|V(\tau,x)|
\\
\le&~ C^2 \left( \frac{\mu(t)}{\mu(\tau)}\right)^{\beta} \mu(t)^{\mathrm{sign}(t)\epsilon} \mu(\tau)^{\mathrm{sign}(\tau)\epsilon} \|x\|
\\
=&~ C^2 \left( \frac{\mu(t)}{\mu(\tau)}\right)^{\beta+\mathrm{sign}(t)\epsilon}\mu(\tau)^{(\mathrm{sign}(t)+\mathrm{sign}(\tau))\epsilon}\|x\|
\\
\le&~ C^2 \left( \frac{\mu(t)}{\mu(\tau)}\right)^{\beta+\mathrm{sign}(t)\epsilon}\mu(\tau)^{2\mathrm{sign}(\tau)\epsilon}\|x\|,
\end{align*}
due to
$$
 \mu(\tau)^{\mathrm{sign}(t)\epsilon} \le \mu(\tau)^{\mathrm{sign}(\tau)\epsilon}, \quad \forall t\ge \tau.
$$
Similarly, for $x\in\mathcal{G}_\tau^u$ and $t\ge \tau$, from \eqref{u-grow} and Definition \ref{str-V}-(1)(3), we see that
\begin{align*}
\|\Psi(t,\tau)x\| \ge &~ \frac{\mu(t)^{-\mathrm{sign}(t)\epsilon}}{C} V(t,\Psi(t,\tau)x)
       \ge \frac{\mu(t)^{-\mathrm{sign}(t)\epsilon}}{C} \left(\frac{\mu(\tau)}{\mu(t)}\right)^\alpha V(\tau,x)
\\
\ge&~ \frac{1}{C^2} \left(\frac{\mu(\tau)}{\mu(t)}\right)^\alpha \mu(t)^{-\mathrm{sign}(t)\epsilon} \mu(\tau)^{-\mathrm{sign}(\tau)\epsilon} \|x\|
\\
=&~ \frac{1}{C^2} \left(\frac{\mu(\tau)}{\mu(t)}\right)^{\alpha-\mathrm{sign}(\tau)\epsilon} \mu(t)^{-(\mathrm{sign}(t)+\mathrm{sign}(\tau))\epsilon}\|x\|.
\end{align*}
It implies that
\begin{align*}
\|\Psi(t,\tau)^{-1}x\|  \le&~ C^2 \left(\frac{\mu(\tau)}{\mu(t)}\right)^{-\alpha+\mathrm{sign}(\tau)\epsilon}
  \mu(t)^{(\mathrm{sign}(t)+\mathrm{sign}(\tau))\epsilon}\|x\|
\\
\le&~ C^2 \left(\frac{\mu(\tau)}{\mu(t)}\right)^{-\alpha+\mathrm{sign}(\tau)\epsilon}  \mu(t)^{2\mathrm{sign}(t)\epsilon}\|x\|,
\end{align*}
since
$$
\mu(t)^{\mathrm{sign}(t)\epsilon} \ge \mu(t)^{\mathrm{sign}(\tau)\epsilon}, \quad \forall t\ge \tau.
$$
This proves the second inequality of \eqref{u-lem}, and the proof of proposition is completed.
\end{proof}

\begin{proof}[Proof of Proposition \ref{thm-ud}.]
It follows from Definition \ref{Lp-func}-(1), there exist subspaces $\mathcal{G}_\tau^s\subset \mathcal{E}_\tau^s$ and $\mathcal{G}_\tau^u\subset \mathcal{E}_\tau^u$, for each $\tau\in\mathbb{R}$, such that $\mathbb{R}^n=\mathcal{G}_\tau^s\oplus \mathcal{G}_\tau^u$.
For each $t\in\mathbb{R}$, define
\begin{align}\label{G-su}
\mathcal{G}_t^s=\Psi(t,\tau)\mathcal{G}_\tau^s, \quad \mathrm{and} \quad \mathcal{G}_t^u=\Psi(t,\tau)\mathcal{G}_\tau^u.
\end{align}
Clearly, $\mathbb{R}^n=\mathcal{G}_t^s\oplus \mathcal{G}_t^u$. Consider the projections:
$$
\pi(t):\mathbb{R}^n\to \mathcal{G}_t^s \quad \mathrm{and} \quad (id-\pi(t)):\mathbb{R}^n\to \mathcal{G}_t^u.
$$
Then we have the following assertion:
\\
{\bf Assertion:} if $V$ is a strict quadratic Lyapunov function and \eqref{S-c} holds, then
\begin{align}\label{pq-est}
\|\pi(t)\|=\|(id-\pi(t))\| \le \sqrt{2}B \mu(|t|)^{2\tilde{\lambda}}
\end{align}
for some constant $B>0$.

In fact, using \eqref{E-su} and \eqref{G-su}, we have that $\mathcal{G}_t^s\subset \mathcal{E}_t^s$ and $\mathcal{G}_t^u\subset \mathcal{E}_t^u$ for each $t\in\mathbb{R}$.
In view of \eqref{QV-def}, we see that
\begin{align}\label{V-st}
\text{$V(t,\pi(t)x)^2 = \langle S(t)\pi(t)x, \pi(t)x \rangle$ \, and \, $V(t,(id-\pi(t))x)^2 = -\langle S(t)(id-\pi(t))x, (id-\pi(t))x \rangle$.}
\end{align}
Set $x=x_s+x_u$ with $x_s=\pi(t)x\in\mathcal{G}_t^s$ and $x_u=(id-\pi(t))x\in\mathcal{G}_t^u$.
Let $\nu(t)>0$ and define
$$
V_s(t,x_s)=-V(t,x_s)^2 + \nu(t)\|x_s\|^2=-\langle S(t)x_s,x_s \rangle + \nu(t)\|x_s\|^2
$$
and
$$
V_u(t,x_u)= V(t,x_u)^2 -\nu(t)\|x_u\|^2=-\langle S(t)x_u,x_u \rangle - \nu(t)\|x_u\|^2.
$$
Then by \eqref{S-c},
$$
V_s(t,x_s) \le -\frac{C_s}{4} \mu(|t|)^{-2\tilde{\lambda}}\|x_s\|^2+\nu(t)\|x_s\|^2 = (\nu(t)-\frac{C_s}{4} \mu(|t|)^{-2\tilde{\lambda}})\|x_s\|^2
$$
and
$$
V_u(t,x_u) \ge \frac{C_u}{4} \mu(|t|)^{-2\tilde{\lambda}}\|x_u\|^2 - \nu(t)\|x_u\|^2 = (\frac{C_u}{4} \mu(|t|)^{-2\tilde{\lambda}}- \nu(t))\|x_u\|^2.
$$
Hence if $\nu(t)\le \max\{C_s/4,C_u/4\}\mu(|t|)^{-2\tilde{\lambda}}$, then
$$
V_s(t,x_s)\le 0 \quad \mathrm{and} \quad V_u(t,x_u)\ge 0,
$$
which implies that
$$
-\left\langle S(t)\pi(t)x,\pi(t)x \right\rangle + \nu(t)\|\pi(t)x\|^2 \le 0
$$
and
$$
-\left\langle S(t)(id-\pi(t))x,(id-\pi(t))x \right\rangle -\nu(t) \|(id-\pi(t))x\|^2 \ge 0
$$
by \eqref{V-st}. Therefore,
\begin{align*}
 & \nu(t)\|\pi(t)x\|^2 + \nu(t) \|(id-\pi(t))x\|^2
\\
&\quad  -\left\langle S(t)\pi(t)x,\pi(t)x \right\rangle + \left\langle S(t)(id-\pi(t))x,(id-\pi(t))x \right\rangle
\\
&=  \nu(t)\|\pi(t)x\|^2 + \nu(t) \|(id-\pi(t))x\|^2 + \langle S(t)x,x \rangle
\\
&\quad   - 2 \langle S(t)\pi(t)x,x \rangle \le 0,
\end{align*}
and thus,
\begin{align*}
& \nu(t) \left\| \pi(t)x-\frac{1}{2\nu(t)}S(t)x \right\|^2 + \nu(t) \left\| (id-\pi(t))x+\frac{1}{2\nu(t)}S(t)x \right\|^2
\\
&= \nu(t)\|\pi(t)x\|^2 + \nu(t) \|(id-\pi(t))x\|^2 + \langle S(t)x,x \rangle
\\
& \quad +\frac{\|S(t)x\|^2}{2\nu(t)} - 2 \langle S(t)\pi(t)x,x \rangle \le \frac{\|S(t)x\|^2}{2\nu(t)},
\end{align*}
which indicates that
\begin{align*}
\|\pi(t)x\|=&~ \left\|\pi(t)x-\frac{1}{2\nu(t)}S(t)x+\frac{1}{2\nu(t)}S(t)x\right\|
\\
\le&~ \left\|\pi(t)x-\frac{1}{2\nu(t)}S(t)x \right\| + \frac{1}{2\nu(t)} \|S(t)x\|
\\
\le&~ \frac{1}{\sqrt{2}} \|S(t)x\| + \frac{1}{2\nu(t)}\|S(t)x\| \le \frac{\sqrt{2}}{\nu(t)} \|S(t)x\|
\end{align*}
and similarly, $\|(id-\pi(t))x\| \le \frac{\sqrt{2}}{\nu(t)} \|S(t)x\|$.
Choosing $\nu(t)=\max\{C_s/4,C_u/4\} \mu(|t|)^{-2\tilde{\lambda}}$, we have
$$
\|\pi(t)\|=\|(id-\pi(t))\| \le \frac{\sqrt{2}}{\max\{C_s/4,C_u/4\}}  \mu(|t|)^{2\tilde{\lambda}}.
$$

We now continue to prove the proposition.
Observe that the subspaces $\mathcal{G}_t^s$ and $\mathcal{G}_t^u$ satisfy all conditions of Proposition \ref{thm-str}.
Then for all $t\ge\tau$,  it follows from  \eqref{u-lem}, \eqref{pq-est} and the fact $\mu(|\tau|)^{\tilde{\lambda}}=\mu(\tau)^{\mathrm{sign}(\tau)\tilde{\lambda}}$ that
\begin{align*}
\|\Psi(t,\tau)\pi(\tau)\| \le&~ \|\Psi(t,\tau)|\mathcal{G}_\tau^s\| \|\pi(\tau)\|
\\
\le&~ C^2 \left( \frac{\mu(t)}{\mu(\tau)}\right)^{\beta+\mathrm{sign}(t)\epsilon}\mu(\tau)^{2\mathrm{sign}(\tau)\epsilon} \cdot \sqrt{2}B\mu(|\tau|)^{2\tilde{\lambda}}
\\
=&~ \sqrt{2}BC^2 \left( \frac{\mu(t)}{\mu(\tau)}\right)^{\beta+\mathrm{sign}(t)\epsilon} \mu(\tau)^{2\mathrm{sign}(\tau)(\epsilon+\tilde{\lambda})}
\end{align*}
and
\begin{align*}
\|\Psi(t,\tau)^{-1}(id-\pi(t))\| \le&~ \|\Psi(t,\tau)^{-1}|\mathcal{G}_t^u\| \|id-\pi(t)\|
\\
\le&~  C^2 \left(\frac{\mu(\tau)}{\mu(t)}\right)^{-\alpha+\mathrm{sign}(\tau)\epsilon}
  \mu(t)^{ 2\mathrm{sign}(t) \epsilon} \cdot \sqrt{2}B\mu(|t|)^{2\tilde{\lambda}}
\\
=&~ \sqrt{2} BC^2 \left(\frac{\mu(\tau)}{\mu(t)}\right)^{-\alpha+\mathrm{sign}(\tau)\epsilon}
  \mu(t)^{ 2\mathrm{sign}(t) (\epsilon+\tilde{\lambda})}.
\end{align*}
Hence,
\eqref{lin-eq} admits a nonuniform $\mu$-dichotomy if
we take
\begin{align*}
&D=\sqrt{2} BC^2,
\quad \nu=2(\epsilon+\tilde{\lambda}), \quad
\omega=2(\epsilon+\tilde{\lambda}),
\\
 &\lambda_s=\beta+\mathrm{sign}(t)\epsilon,
\quad  \lambda_u=-\alpha+\mathrm{sign}(\tau)\epsilon.
\end{align*}
This completes the proof.
\end{proof}

\section*{Acknowledgments}

This paper was jointly supported by the Natural Science Foundation of Zhejiang Province (No. LZ24A010006),
and the National Natural Science Foundation of China (No. 11671176, 11931016).

\end{document}